\documentclass[a4paper,leqno]{amsart} 

\usepackage{graphicx}
\usepackage{amsmath,amssymb,amsfonts,latexsym,amsthm}
\usepackage{url}
\usepackage{dsfont}
\usepackage{microtype}
\usepackage{textcomp}
\usepackage[active]{srcltx}
\usepackage{enumitem}

\usepackage[top=3cm, left=3cm, right=3cm, bottom =3cm]{geometry}

\usepackage[pdftex,colorlinks=true,urlcolor=blue,citecolor=blue, linkcolor=blue]{hyperref}

\usepackage{cleveref}

\newtheorem{thm}{Theorem}[section]
\newtheorem{lemma}[thm]{Lemma}
\newtheorem{example}[thm]{Example}
\newtheorem{proposition}[thm]{Proposition}
\newtheorem{definition}[thm]{Definition}

\newtheorem{rem}[thm]{Remark}

\newtheorem{note}[thm]{Note}

\numberwithin{equation}{section}

\newcommand{\filtration}{\left(\mathcal{F}_t\right)_{t\in[0,T]}}
\newcommand{\F}{\mathcal{F}}

\newcommand{\filt}{\mathbb{F}}
\newcommand{\essinf}{\mathop{\mbox{essinf}}}
\newcommand{\esssup}{\mathop{\mbox{esssup}}}

\newcommand {\R}{\mathbb {R}}
\newcommand {\N}{\mathbb {N}}
\newcommand {\Q}{\mathbb {Q}}

\newcommand{\eps}{{\varepsilon}}
\newcommand{\wt}{\widetilde}

\newcommand{\beao}{\begin{eqnarray*}}
\newcommand{\eeao}{\end{eqnarray*}\noindent}

\newcommand{\beam}{\begin{eqnarray}}
\newcommand{\eeam}{\end{eqnarray}\noindent}

\makeatletter

\date{}
\title[Nash equilibria for game contingent claims]{Nash equilibria for game contingent claims with utility-based hedging}

\author[Klebert Kentia]{Klebert Kentia}
\author[Christoph K\"uhn]{Christoph K\"uhn}
\address[K. Kentia, C. K\"uhn]{Institut f\"ur Mathematik, Goethe-Universit\"at Frankfurt, D-60054 Frankfurt a.M., Germany}
\email{kentia\,@\,math.uni-frankfurt.de, ckuehn\,@\,math.uni-frankfurt.de}

\begin{document}

\begin{abstract}
Game contingent claims (GCCs) generalize American contingent claims in allowing the writer to recall the option as long as it is not exercised, at the price of paying some penalty.
In incomplete markets, an appealing approach is to analyze GCCs like their European and American counterparts by solving option holder's and writer's optimal investment problems in the underlying securities.
By this, partial hedging opportunities are taken into account. We extend results in the literature by solving the stochastic game corresponding to GCCs with both continuous time stopping 
and trading. Namely, we construct Nash equilibria by rewriting the game 
as a non-zero-sum stopping game in which players compare payoffs in terms of their exponential utility indifference values. 
As a by-product, we also obtain an existence result for the optimal exercise time of an American claim under utility indifference 
valuation by relating it to the corresponding nonlinear Snell envelope.
\end{abstract}
\maketitle

\section{Introduction}
A {\em game contingent claim (GCC)} as introduced  in Kifer~\cite{Kifer2000}, is a contract between a buyer/holder and a seller/writer which can be exercised by the buyer and recalled by the seller 
at any time up to a maturity date  when the contract is terminated anyway. 
The contract can be modeled by two stochastic processes $(X_t)_{t\in[0,T]}$ and $(Y_t)_{t\in[0,T]}$, where $T\in\R_+$ is the maturity:
If the buyer chooses a $[0,T]$-valued stopping time~$\tau$ and the seller a $[0,T]$-valued stopping time $\sigma$, the payoff to the buyer at time $\min\{\tau,\sigma\}$ is given by  
\begin{equation*}
	X_\tau\mathds{1}_{\{\tau\le \sigma\}}+Y_\sigma\mathds{1}_{\{\sigma< \tau\}}.
\end{equation*}
A standing assumption is $X\le Y$, meaning that stopping is penalized.
In the last two decades, such contracts have been intensively studied in the literature. We refer to Kifer~\cite{kifer.suryev} for a recent review.
The starting point was the article by Kifer~\cite{Kifer2000} who showed that in a complete market, the option has a unique no-arbitrage price and can be perfectly hedged by the writer. 
Here, a hedge consists of a dynamic trading strategy in the underlyings and a recalling time. In incomplete markets, however, perfect hedges may fail to exist and there are essentially three 
different approaches generalizing \cite{Kifer2000}. 
First there is super-hedging (see, e.g., Section~13.2.1 of \cite{KallsenKuehn2005}), which has the drawback that it is often too expensive and in many situations only leads to trivial bounds for 
arbitrage-free prices.
Another approach consists in considering GCCs as liquid securities which can be dynamically traded simultaneously with the underlyings, with the only difference that their prices are not 
exogenously given. The exercise and recall features of a GCC are implicitly modeled  by short selling and long buying constraints, respectively. It turns out that both a no-arbitrage 
criterion and a utility maximization criterion for a representative investor lead to the dynamic value of a zero-sum stopping game as a price process for a GCC.
In the first case, expectations are taken under an arbitrary martingale measure and in the second under the  martingale measure induced by the marginal utility at the optimal terminal wealth when 
only trading in the underlyings is possible (see \cite{KallsenKuehn2005} and \cite{KallsenKuehn2004}, respectively). 

In the present article, we follow the third approach which is in the spirit of utility-based hedging of European claims as introduced by Hodges and Neuberger~\cite{HodgesNeuberger89}, see the survey article 
Becherer~\cite{Becherer.Survey2010} and the references therein.
For an analysis of the general indifference valuation problem for American contingent claims, we refer to Leung and Sircar~\cite{leung.sircar.2009} for a backward utility-based approach as used in the present paper, and to Leung et al.~\cite{LeungSircarZhariphopoulou2012} who apply a  forward performance criterion approach to the utility-based hedging problem.
%
%
Here, we consider as in \cite{Kuehn2004} a game between the buyer and the seller of a GCC who both aim to maximize their expected utility from terminal wealth by exercising/recalling the GCC and trading, in addition, arbitrary amounts in the underlyings.
This means that as in \cite{Kifer2000}, but in contrast to \cite{KallsenKuehn2004} 
and \cite{KallsenKuehn2005}, the GCC is no liquid asset. 
Applied to the special case of a complete market, the game leads to the same stopping times as in \cite{Kifer2000}, i.e., the buyer tries to maximize, whereas the seller 
tries to minimize the expected option's payoff under the unique equivalent martingale measure. Especially, this means that the equilibrium stopping times do not depend 
on agents' preferences or endowments, and equilibrium values are unique. In \cite{Kuehn2004}, it is furthermore  shown that Nash equilibria exist
also in a general incomplete market if the utility functions are exponential, i.e.,
the absolute risk aversion does not depend on wealth. However, while trading in the underlying is continuous, stopping the contract is only discrete in \cite{Kuehn2004}. This allows to construct equilibrium stopping regions 
by a backward recursion in time.
Later on, in the seminal paper of Hamad\`ene and Zhang~\cite{HamadeneZhang2010}, equilibria of non-zero-sum continuous time stopping games 
have been derived under minimal conditions on the payoff processes. This can be applied to the above game 
for arbitrary utility functions of the holder and the writer, but only in the special case that they do not have access to a financial market, i.e., (partial) hedging opportunities 
are not taken into account (see Section~4 of \cite{HamadeneZhang2010}). Since the utility functions are typically nonlinear, it is a non-zero-sum game and equilibria are in general not unique. 

An observation in \cite{Kuehn2004} is that Nash equilibrium points may fail to exist for players' utility functions other than exponential. Namely, by the non-constancy of the 
absolute risk aversion, equilibria cannot be constructed backwards in time since the past trading gains in the underlyings matter (cf.\  Remarks~2.4 and~2.5 therein for a counterexample and a detailed explanation, respectively).
This also shows that the result of \cite{HamadeneZhang2010}, which is {\em not} restricted to 
exponential utility, cannot be applied to the problem if there are (partial) hedging opportunities.
In the current article, we close the above mentioned gap and extend \cite{Kuehn2004} for exponential utility functions to continuous time stopping
(see \cref{thm:NEPforGamePb}). For this, we combine the techniques of derivation of the above mentioned result of \cite{HamadeneZhang2010}, who consider games under linear expectations, with new results on 
optimal stopping when payoffs are valued by utility indifference that is in general nonlinear in the payoff.
We note that the results of \cite{HamadeneZhang2010} have recently been partly extended by Grigorova and Quenez~\cite{GrigorovaQuenez2017} to a non-zero-sum game with players evaluating their 
payoffs in terms of (nonlinear) $g$-expectations for Lipschitz generator functions $g$. However, although in a continuous time setup, the latter article restricts the equilibrium analysis of 
the game to discrete time stopping strategies. Furthermore since the dynamic exponential indifference valuation typically is a $g$-expectation with $g$ of quadratic 
growth (cf.\ e.g.\ Theorem 13 in Mania and Schweizer~\cite{ManiaSchweizer05}), the results of \cite{GrigorovaQuenez2017} are not directly applicable to our non-zero-sum game.

Given a recalling time of the option writer, for an exponential utility function, the utility maximization and timely exercise problem of the option holder can be reduced to an optimal stopping problem in 
which the random payoff of the American claim is not evaluated by its (linear) expectation, but by its (buyer's) indifference price that is not homogeneous.
This means that one has to solve $\sup_{\tau}\, \pi_0(L_\tau)$, where $L$ is the payoff process, $\pi_0$ is the initial indifference valuation, and  $\tau$ runs through all $[0,T]$-valued 
stopping times the holder can choose. Let $\pi_t$ be the conditional indifference valuation at time $t$. By the time-consistency of the indifference valuation operator $\pi=(\pi_t)_{t\in[0,T]}$, 
it seems to be self-evident that there should be a smallest ``$\pi$-supermartingale'' that dominates $L$ and, if $L$ has no negative jumps, the optimal stopping time is given by the first 
time this supermartingale hits $L$. However, in continuous time, it seems very difficult to provide rigorous proofs to this conjecture. 
Our result on optimal stopping under indifference valuation is \cref{thm:NonlinSnellNEW}, which characterizes the optimal exercise time of an American claim as the first time the payoff process hits the 
corresponding nonlinear Snell envelope associated to the American exponential utility indifference value. To the best of our knowledge this is a new result, and it is also of independent interest. 
For its proof, we extend properties of the dynamic {\em European} indifference valuation derived by Mania and Schweizer~\cite{ManiaSchweizer05} to the American one
and establish a reverse continuity result (\cref{pro:ConvOfExpIMpliesConvInProb}) for the indifference valuation. 
How the result is related to the literature on optimal stopping under nonlinear expectations is described in subsection~\ref{14.6.2018.1}.
%
%

In incomplete markets, random endowments are a key motivation to trade derivatives.
For European claims, Anthropelos and \v{Z}itkovi\'c~\cite{anthropelos.zitkovic} give a complete characterization of those claims for which there exists a price at which two agents, with given risk-aversions and endowments, are willing to trade the claim. 
It is shown that there exists a unique, up to replicable payoffs, Pareto-optimal allocation (see Remark~3.17 in \cite{anthropelos.zitkovic}). The special feature of exponential utility
maximizers is that by the constant absolute risk-aversion, the risky-sharing
does not depend on the distribution of the aggregate endowment between the agents before trading. Consequently, the above mentioned Pareto-optimal allocation can be reached by a ``mutually agreeable''  trade for {\em any} initial distribution of 
the aggregate endowment between the agents. In a similar vein, agents' endowments also influence the optimal stopping times of a given GCC, and thus the option's payoff.
In section~\ref{14.6.2018.2}, we illustrate this impact and discuss the relation to \cite{anthropelos.zitkovic}, see Example~\ref{9.5.2018.1} and Remark~\ref{5.9.2018.1}, respectively.
Example~\ref{5.9.2018.2} provides some economic intuition behind the non-uniqueness of equilibria that is caused by a simultaneous incentive of both agents to stop the contract. By contrast, in complete markets, both players evaluate future payoffs by their conditional expectation under the unique martingale measure, which implies
by $X\le Y$, that at no time, they both want to stop the contract.\\

The paper is organized as follows. In \cref{sect:ProblemFormulationUtilMax}, we specify the mathematical framework and state the main result, \cref{thm:NEPforGamePb}, about the 
existence of Nash equilibria. By relating the problem to utility-indifference valuation, \cref{2.2.2017.1} prepares the proof of \cref{thm:NEPforGamePb} which is then completed 
in \cref{sect:proofMainResult}. In addition, \cref{2.2.2017.1} states the key \cref{thm:NonlinSnellNEW} on optimal stopping under utility indifference valuation, which is accompanied by a literature review on nonlinear optimal stopping 
(subsection~\ref{14.6.2018.1}).
Examples can be found in section~\ref{14.6.2018.2}.
Finally, \cref{app:AppendixB} gathers proofs of results that are omitted throughout the main text.

\section{Problem formulation and main result}\label{sect:ProblemFormulationUtilMax}
We start with a filtered probability space $(\Omega,\F,\filt=\filtration,P)$ with time horizon $T\in\R_+$ and a filtration $\filt$ satisfying the usual conditions of right-continuity and completeness.
For a $[0,T]$-valued stopping time $\tau$, we denote by $\mathcal{T}_\tau$ the family of stopping times $\sigma$ such that $\tau\le \sigma\le T$ $P$-almost-surely. 
We denote by $Z^Q$ the density process of an equivalent measure $Q$ with respect to $P$ and by $E^Q_\tau[\cdot]$ the conditional expectation under $Q$ given the information $\F_\tau$ 
up to a stopping time $\tau\in\mathcal{T}_0$. In particular for $Q=P$, we simply write $E_\tau[\cdot]$. If not stated otherwise, inequalities between random variables 
are understood in the $P$-almost-surely sense.
For a generic $\sigma$-algebra $\mathcal{A}$ on $\Omega$ and a probability measure $Q$ on $(\Omega,\mathcal{A})$, we denote by $L^\infty(\mathcal{A},Q)$ the space of $\mathcal{A}$-measurable random variables that are $Q$-essentially bounded and by $\mathcal{S}^\infty(Q)$ the space of 
$\filt$-adapted processes $Y$ with  c\`adl\`ag paths satisfying 
	$\lVert Y\rVert_{\mathcal{S}^\infty(Q)}:= \lVert \sup_{t\in[0,T]}\lvert Y_t\rvert\rVert_{L^\infty(\F_T, Q)}<\infty.$
In particular, we simply write $L^\infty(\mathcal{A})$ for $L^\infty(\mathcal{A},P)$, $L^\infty, \mathcal{S}^\infty$ for $L^\infty(\F_T,P),\mathcal{S}^\infty(P)$, and $\lVert\cdot\rVert_\infty$ for $\lVert \cdot\rVert_{L^\infty}$.

We consider a general, possibly incomplete, financial market consisting of $d$ underlying risky assets with discounted price process $S=(S^i)_{i=1,\ldots,d}$ being a semimartingale, and a 
riskless asset with unit discounted price. Throughout, we assume that the risky asset price process 
\begin{equation}\label{eq:StockLocallyBdd}
	S\ \text{is locally bounded}
\end{equation}
(for the main result of the article, it has to be even continuous). Denoting $\mathcal{M}^e:=\mathcal{M}^e(S,P)$ the set of equivalent local martingale measures for $S$, we assume 
there exists at least one element $Q$ of $\mathcal{M}^e$ that has finite entropy $E\left[Z^Q_T\log Z^Q_T\right]$ with respect to $P$ in the sense that 
\begin{equation}\label{eq:NonEmptyELMMfE}
	\mathcal{M}^e_f:=\mathcal{M}^e_f(P):=\left\lbrace Q\in \mathcal{M}^e(P)\, \left\lvert\, E^P\left[Z^Q_T\log Z^Q_T\right]<\infty\right.\right\rbrace\neq \emptyset.
\end{equation}
In particular, there exists a unique measure $Q^E\in\mathcal{M}^e_{f}$, the so-called 
\emph{entropy minimizing martingale measure} (EMMM), that satisfies 
\begin{equation}\label{8.3.2017.2}
E^{Q^E}\Big[\log Z^{Q^E}_T\Big] = \inf_{Q\in\mathcal{M}^e_{f}} E^Q\Big[\log Z^Q_T\Big],
\end{equation}
cf.\ Theorem~2.1 in Frittelli~\cite{Frittelli2000}, which beyond boundedness of $S$ also extends to locally bounded $S$. 
We denote by $L(S)$ the space of $\filt$-predictable $S$-integrable $\R^d$-valued processes. For $\vartheta\in L(S)$, the stochastic integral of $\vartheta$ with respect to $S$ is 
denoted $\int_0^\cdot\vartheta^{\mathrm{tr}}_s\,dS_s$.
We work as in \cite{ManiaSchweizer05} with a space of admissible trading strategies
\begin{equation}\label{eq:DefStrategiesTheta}
	\Theta = \left\lbrace \vartheta\in L(S)\,\Big\lvert\,\Big. \begingroup\textstyle\int_0^\cdot\vartheta^{\mathrm{tr}}_sdS_s\endgroup\text{ is a }Q\text{-martingale for all }Q\in \mathcal{M}^e_f\right\rbrace .
\end{equation}
Note that such admissible trading strategies clearly exclude arbitrage opportunities. 

Consider two agents, $A$ and $B$, who are the seller and buyer of a GCC, respectively. We assume that $A$ and $B$ in addition to entering the contract, have access to the financial market. 
The agents' preferences are modeled by exponential utility functions $U_A, U_B$, with constant absolute risk-aversion parameters $\alpha_A,\alpha_B>0$, i.e., 
$U_A(x)=-\exp(-\alpha_A x)$ and $U_B(x)=-\exp(-\alpha_B x),\quad x\in\R$.
Let $X,Y\in\mathcal{S}^\infty$ and define 
\begin{equation}\label{eq:PayoffProcessesGGC}
	R(\tau,\sigma):= X_\tau\mathds{1}_{\{\tau\le \sigma\}}+Y_\sigma\mathds{1}_{\{\sigma< \tau\}}
\end{equation}
the GCC payoff of agent $B$ paid by agent $A$ at time $\tau\wedge \sigma,$ when $A$ and $B$ choose stopping strategies $\sigma$ and $\tau$, respectively, for $\tau,\sigma\in\mathcal{T}_0$.
Agents $A$ and $B$ have exogenous endowments given by the contingent claims $C_A,C_B\in L^\infty$, which are in general not replicable by trading in the financial market. 
By the randomness of $C_A$ and $C_B$, indifference valuations of the agents may also depend on their respective endowments. We are interested in Nash equilibrium points (see \cref{def:NEPDefnUtilMax}) 
in the stopping time strategies for the buyer and seller whose objectives are as follows. The seller $A$ wants to maximize her expected utility from terminal wealth $C_A-R(\tau,\sigma)+\int_{0}^T\vartheta^{\mathrm{tr}}_s\,dS_s$ after entering the contract at 
time $t=0$ and trading in the financial market according to a self-financing strategy $\vartheta=(\vartheta^i)_{i=1}^d$ in $\Theta$, with $\vartheta^i_t$ denoting the number of shares of 
asset $i$ held at time $t$, $t\in[0,T]$.
The corresponding maximization problem for the seller is 
\begin{equation}\label{eq:Defnu1}
	u_A(\tau,\sigma):=\sup_{\vartheta\in\Theta}E\left[-\exp\left(-\alpha_A\,\left(C_A-R(\tau,\sigma)+\int_{0}^T\vartheta^{\mathrm{tr}}_s\,dS_s\right)\right)\right]. 
\end{equation}
Similarly, the buyer of the contract wants to maximize her expected utility from terminal wealth $C_B + R(\tau,\sigma)+\int_{0}^T\vartheta^{\mathrm{tr}}_s\,dS_s$ and her 
maximization problem is 
\begin{equation}\label{eq:Defnu2}
	u_B(\tau,\sigma):=\sup_{\vartheta\in\Theta}E\left[-\exp\left(-\alpha_B\,\left(C_B+R(\tau,\sigma)+\int_{0}^T\vartheta^{\mathrm{tr}}_s\,dS_s\right)\right)\right]. 
\end{equation}
\begin{definition}\label{def:NEPDefnUtilMax}
 We say that a pair $(\tau^*,\sigma^*)\in\mathcal{T}_0\times\mathcal{T}_0$ is a Nash equilibrium point~(NEP) for the non-zero-sum game associated to \cref{eq:Defnu1,eq:Defnu2}
 if 
 \begin{equation}\label{eq:NEPdefn}
 	u_A(\tau^*,\sigma^*) \ge u_A(\tau^*,\sigma)\quad \text{and}\quad u_B(\tau^*,\sigma^*) \ge u_B(\tau,\sigma^*)\quad \forall (\tau,\sigma)\in\mathcal{T}_0\times\mathcal{T}_0.
 \end{equation}
\end{definition}
\begin{rem}
Alternatively, one may model the problem as a so-called extensive game in which players' decisions are sequential, and each player makes a decision depending on the ``nature'' (given by the filtration) 
and the past actions of her counter-party.  
The sequential decisions consist, at each step, of stopping the contract or not, and choosing the amount of underlyings held in the portfolio.
We refer to Gonz\'alez-D\'iaz et al.~\cite{intro.game} for an introduction to extensive games and related concepts.
At least in finite discrete time and finite~$\Omega$, one can easily prove that a NEP in the sense of \cref{eq:NEPdefn} induces a stochastic feedback Nash equilibrium in the extensive game described above. 
Namely, given a NEP \cref{eq:NEPdefn}, one takes $\tau^*,\sigma^*$ together with 
the investment strategies which attain the suprema \cref{eq:Defnu1,eq:Defnu2} for $\tau=\tau^*$ and $\sigma=\sigma^*$, and 
consider them as feedback strategies where the response function
is degenerated, i.e., each player simply ignores past actions of her counter-party. Only past actions of the nature are used because the quantities are in general 
stochastic. Note that given $(\tau^*,\sigma^*)$, the suprema \cref{eq:Defnu1,eq:Defnu2} can be determined separately since the investment strategy of one player 
has no effect on the wealth of the other player (trading has, e.g., no impact on the underlying's price). Since, in addition, the option's payoff 
cannot be influenced anymore after the first player stops the contract,
and, consequently, $\tau^*,\sigma^*$ are stopping strategies that implicitly condition that the other player has not stopped yet, 
the ``feedback'' strategies constructed above with degenerated response functions are also optimal in the set of all feedback strategies with arbitrary response functions. 
Furthermore, the same holds a fortiori when looking at a variant of this extended game in which each player only observes the stopping time but not the investment strategy of the
other player. Finally, since the NEPs are constructed backwards in time
(see the construction (2.3)-(2.7) in \cite{Kuehn2004} for the special case that stopping is discrete),
there exists a subgame perfect equilibrium. 
We leave it as an easy exercise for the reader to write down the finite extensive game
and prove the above assertions (since the players may stop 
simultaneously, one has to work with a nontrivial information partition in Definition~3.1.1 of \cite{intro.game}). Since the finite extensive game  
boils down to the game \cref{eq:Defnu1}/\cref{eq:Defnu2}, for the
continuous time modeling, we prefer to start directly with \cref{eq:Defnu1}/\cref{eq:Defnu2}.
\end{rem}
The main result of the article is the following theorem, and the proof is deferred to \cref{sect:proofMainResult}.
\begin{thm}\label{thm:NEPforGamePb}
Assume that 
\begin{equation}\label{eq:aspFiltContinuous}
 	\text{the filtration }\filt \text{ is continuous},
\end{equation}
i.e., any local $\filt$-martingale is $P$-a.s.\ continuous, and the payoff processes satisfy
\begin{equation}\label{eq:XleY}
	X,Y\in\mathcal{S}^\infty\ \text{with}\ X_t\le Y_t,\ t\in[0,T],\ P\text{-a.s.},
\end{equation}
and
\begin{equation}\label{eq:NoNegJumps}
\text{$X$ has only nonnegative jumps, and $Y$ has only nonpositive jumps.}
\end{equation}		
Then, the non-zero-sum game associated to $u_A, u_B$ given in \cref{eq:Defnu1,eq:Defnu2} with $\alpha_A,\alpha_B\in(0,\infty)$, $C_A,C_B\in L^\infty$ admits a NEP $(\tau^*,\sigma^*)\in\mathcal{T}_0^2$. 
\end{thm}

\begin{rem} The difference $Y-X\ge0$ can be interpreted as the penalty the writer of the 
GCC has to pay if recalling the option before it is exercised. 
Consequently, there is a negative attitude towards stopping as for both players 
it appears more advantageous that her counter-party terminates the game in her stead. 
Condition \cref{eq:NoNegJumps} guarantees that optimal stopping times are attained, and we need not deal with almost optimal stopping times.
\end{rem}

\section{Indifference value}\label{2.2.2017.1}
In this section, we first state old and new facts on exponential utility indifference valuation. Then, we characterize a NEP in terms of indifference values
for the payoffs \cref{eq:PayoffProcessesGGC}. Finally, we characterize the optimal 
stopping time when payoffs are evaluated at their indifference prices. 
This prepares the  proof of our main result \cref{thm:NEPforGamePb} concerning existence of a NEP for the game \cref{eq:Defnu1}/\cref{eq:Defnu2},
that we provide in \cref{sect:proofMainResult}.

The game problem  \cref{eq:NEPdefn} can be reformulated in terms of the utility indifference valuation of suitable claims, in a way that one obtains a non-zero-sum Dynkin game 
in which players evaluate payoffs directly by utility indifference. Before making this relation precise, we recall briefly the definition and some dynamic properties of the exponential 
indifference valuation for bounded claims. A relatively general study of this was performed by \cite{ManiaSchweizer05} in a setup analogous to 
the one of \cref{sect:ProblemFormulationUtilMax}, with $\Theta$ corresponding exactly to the space $\Theta_2$ of trading strategies in Delbaen et al.~\cite{DelbaenetAl2002}, where other possible spaces of trading strategies are also
compared. As we have adopted an analogous setup, we are able to use some results of \cite{ManiaSchweizer05} when the need arises.

Let $H$ be a contingent claim in $L^\infty(\F_T)$ and consider an agent who is willing to buy the claim $H$ at some time $t\in[0,T]$, and whose preferences are described by an exponential utility function $U$
for some risk-aversion parameter $\alpha\in(0,\infty)$, i.e.\  $U(x)=-e^{-\alpha x},\ x\in\R$. In addition, the agent has the random endowment~$C\in L^\infty(\F_T)$. 
The buyer's indifference value $\pi^{\alpha,C}_t(H)$ at time $t\in[0,T]$ is the amount that the agent needs to pay 
at time $t$ to receive the claim $H$ at terminal time $T$ so that the agent's maximal expected utility from additionally trading between $t$ and $T$ with zero initial capital
coincides with his maximal expected utility from solely trading with zero initial capital. In other words, $\pi^{\alpha,C}_t(H)$ is the amount that makes an agent indifferent between 
buying or not buying the claim $H$ and optimally trading in the financial market. More precisely for $H\in L^\infty$ and $t\in[0,T]$, the indifference value $\pi^{\alpha, C}_t(H)$ of $H$ at time $t$ is implicitly given by
\begin{equation}\label{eq:DefnIndPrice}
	\esssup_{\vartheta\in\Theta} E_t\Big[-e^{-\alpha \left(C+\int_t^T\vartheta^{\mathrm{tr}}_sdS_s\right)}\Big] = \esssup_{\vartheta\in\Theta} E_t\Big[-e^{-\alpha \left(C+H-\pi^{\alpha,C}_t(H)+\int_t^T\vartheta^{\mathrm{tr}}_sdS_s\right)}\Big].
\end{equation}
Under \cref{eq:StockLocallyBdd,eq:NonEmptyELMMfE}, Theorem~2.2 in \cite{DelbaenetAl2002} (whose condition (2.13) was shown by Kabanov and Stricker~\cite{KabanovStricker2002} to be unneeded) and a 
dynamic programming principle imply that the left-hand side of \cref{eq:DefnIndPrice} almost surely does not vanish, and hence a direct reformulation of \cref{eq:DefnIndPrice} yields 
\begin{equation}\label{eq:IndPricePrimalCharact}
	\pi^{\alpha,C}_t(H) = - \frac{1}{\alpha}\log\left(\frac{\displaystyle\esssup_{\vartheta\in\Theta} E_t\Big[-e^{-\alpha\big(C+H+\int_t^T\vartheta^{\mathrm{tr}}_sdS_s\big)}\Big]}{\displaystyle\esssup_{\vartheta\in\Theta} E_t\big[-e^{-\alpha\big(C+\int_t^T\vartheta^{\mathrm{tr}}_sdS_s\big)}\big]}\right).
\end{equation}
In general, the indifference value depends on the exogenous random endowment~$C$, but by replacing the measure $P$ with $P_C$ defined via
$dP_C/dP := \exp(-\alpha C)/E[\exp(-\alpha C)]$, one can reduce it to the case without a random endowment. In particular $\Theta$ and $\mathcal{M}^e_f$ remain the same when changing from 
$P$ to $P_C$. Note that 
\begin{equation*}
\pi^{\alpha,C}_t(H) = \pi^{\alpha,0}_t(C+H) - \pi^{\alpha,0}_t(C),
\end{equation*}
which immediately follows from \cref{eq:IndPricePrimalCharact}. 
In addition by Proposition~2 of \cite{ManiaSchweizer05}, one also has a dual representation
\begin{equation*}
\pi^{\alpha,0}_t(H) = \essinf_{Q\in\mathcal{M}^e_{f}}\left(E^Q_t[H] + \frac{1}{\alpha}\left(E^Q_t\Big[\log\frac{Z^Q_T}{Z^Q_t}\Big]-\alpha \essinf_{Q\in\mathcal{M}^e_{f}} E^Q_t\Big[\frac{1}{\alpha}\log\frac{Z^Q_T}{Z^Q_t}\Big]\right)\right).
\end{equation*}
In the following proposition, we state some dynamic properties of the exponential indifference value.
\begin{proposition}\label{pro:IndiffValueProperties}
	For $\alpha\in(0,\infty)$ and $C\in L^\infty$,  $(\pi^{\alpha,C}_t)_{t\in[0,T]}$ defines mappings $\pi^{\alpha,C}_t:L^\infty(\F_T)\ni H\mapsto \pi^{\alpha,C}_t(H)\in L^\infty(\F_t)$ satisfying
\begin{enumerate}
	\item \emph{``C\`adl\`ag version''}:  For $H\in L^\infty$, $\lVert \pi^{\alpha,C}_t(H)\rVert_\infty\le \lVert H\rVert_\infty,\ t\in[0,T],$ and there exists a c\`adl\`ag process $\Gamma^H$ with $\Gamma^H_t=\pi^{\alpha,C}_t(H)\ P$-a.s.\ for every $t\in[0,T]$ and
	$$
	\Gamma^H_\tau = -\frac{1}{\alpha}\log\,\essinf_{\vartheta\in\Theta}\, E^{Q^{E,C}}_\tau\Big[e^{-\alpha\left(H+\int_\tau^T\vartheta^{\mathrm{tr}}_sdS_s\right)}\Big]=:\pi^{\alpha,C}_\tau(H),\quad \text{for }\tau\in\mathcal{T}_0,
	$$
	where $Q^{E,C}$ is EMMM from \cref{8.3.2017.2} after replacing $P$ with $P_C$. 
	\item \emph{``(Strict) monotonicity''}: If $H^1\le H^2$ then $\pi^{\alpha,C}_\tau(H^1)\le \pi^{\alpha,C}_\tau(H^2)$ for all $\tau\in\mathcal{T}_0$. If in addition $\pi^{\alpha,C}_0(H^1)= \pi^{\alpha,C}_0(H^2)$,
	then $H^1=H^2$.
	\item \emph{``Replication invariance''}: $\pi^{\alpha,C}_\tau\left(H+x_\tau+\int_\tau^T\vartheta^{\mathrm{tr}}_sdS_s\right) = \pi^{\alpha,C}_\tau(H)+x_\tau$, for any $H\in L^\infty,\ \tau\in\mathcal{T}_0,\ x_\tau\in L^\infty(\F_\tau),\ \vartheta\in\Theta$.
	\item \emph{``Replication cost preservation''}: $\pi^{\alpha,C}_\tau\left(x_\tau+\int_\tau^T\vartheta^{\mathrm{tr}}_sdS_s\right) = x_\tau$, for 
	any $\tau\in\mathcal{T}_0$,\ $x_\tau\in L^\infty(\F_\tau),\ \vartheta\in\Theta.$
	\item \emph{``Local property''}: $\pi^{\alpha,C}_\tau(H^1\mathds{1}_\Lambda + H^2\mathds{1}_{\Lambda^c}) = \mathds{1}_\Lambda\pi^{\alpha,C}_\tau(H^1) + \mathds{1}_{\Lambda^c}\pi^{\alpha,C}_\tau(H^2)$, for any $H^1,H^2\in L^\infty,\ \tau\in\mathcal{T}_0,\ \Lambda\in\F_\tau$.
	\item \emph{``(Stopping) time consistency''}: $\pi^{\alpha,C}_\tau(H) = \pi^{\alpha,C}_\tau(\pi^{\alpha,C}_\sigma(H))$, for any $H\in L^\infty$, $\tau\in\mathcal{T}_0$ with $\sigma\in\mathcal{T}_\tau$.
	\item \emph{``Continuity''}: If $\filt$ is continuous, then for any sequence $(H^n)_{n\in\N}$ bounded in $L^\infty$ that converges in probability to some $H\in L^\infty$ as $n\to\infty$, one has
	\[
		\sup_{t\in[0,T]}\lvert \pi^{\alpha,C}_t(H^n)-\pi^{\alpha,C}_t(H)\rvert\longrightarrow 0\quad \text{in probability as }n\to\infty.
	\]
\end{enumerate}
\end{proposition}
The proofs of these properties can be either found in Propositions~4, 12, 14, and 15 of Mania and Schweizer~\cite{ManiaSchweizer05} or are straightforward generalizations. Note that only the case $C=0$ has to be considered, since the extension to general $C\in L^\infty$ is straightforward by replacing $P$ with $P_C$.
For the strict monotonicity in 2., one uses that the supremum in the numerator of \cref{eq:IndPricePrimalCharact} is attained, see Theorem~2.2 in \cite{DelbaenetAl2002}.

The following result can be seen as a reverse of part 7.\ above. 
\begin{proposition}\label{pro:ConvOfExpIMpliesConvInProb}
Let $\filt$ be continuous, $\alpha\in(0,\infty)$ and $C\in L^\infty$. Let $(H^n)_{n\in\N}$ and $(\eta^n)_{n\in\N}$ be bounded sequences in $L^\infty$ such that $\eta^n\le 0\ P\text{-a.s.}$ for all $n\in\N$ and 
$\pi^{\alpha,C}_0(H^n+\eta^n)-\pi^{\alpha,C}_0(H^n)\longrightarrow 0$, as $n\to\infty$. Then 
$\eta^n\rightarrow 0$ in probability, as $n\to \infty.$ 
\end{proposition}
The result in \cref{pro:ConvOfExpIMpliesConvInProb} is new to the best of our knowledge, and constitutes a crucial ingredient in the proofs of both 
\cref{thm:NonlinSnellNEW,lem:2ndEqTowardsNEP}, needed for achieving our main result \cref{thm:NEPforGamePb}.
We include the proof in \cref{app:AppendixB}.
\begin{note}\label{15.3.2017.1}
Let $\alpha>0$. There exists a market model with a sequence of nonnegative 
claims~$(H^n)_{n\in\N}\subset L^\infty$ such that $P[\liminf_{n\to\infty} H^n = \infty]>0$, but the sequence $\pi^{\alpha,0}_0(H^n)$ of indifference values is uniformly bounded.
\end{note}
\begin{proof}[Proof of \cref{15.3.2017.1}]
Consider a market model with $S=1$, i.e., there are no hedging instruments
and take some $A\in\mathcal{F}$ with $P[A]\in (0,1)$.
For the nondecreasing sequence of European claims $H^n:=n \mathds{1}_A$,\ $n\in\N$, the indifference values~$\pi^{\alpha,0}_0(H^n)$ satisfy $\exp(-\alpha\,\pi^{\alpha,0}_0(H^n)) = \exp(-\alpha n)P[A] + 1-P[A]$, and thus 
$\pi^{\alpha,0}_0(H^n)$ increases to the finite constant $-\ln(1-P[A])/\alpha$, as $n$ tends to infinity.
\end{proof}

In the sequel, we denote by 
\begin{equation}\label{26.7.2017.1}
\pi^A:=\pi^{\alpha_A,C_A}\qquad\mbox{and}\qquad \pi^B:=\pi^{\alpha_B,C_B}
\end{equation}
the exponential utility indifference valuation operators for agents $A$ and $B$, respectively. 
In terms of indifference values, a NEP for the game \cref{eq:Defnu1}/\cref{eq:Defnu2} can be characterized as follows.
\begin{proposition}\label{pro:NEPinTermsofIndiffPrice}
	Let $X,Y\in \mathcal{S}^\infty$. A pair $(\tau^*,\sigma^*)\in\mathcal{T}_0\times\mathcal{T}_0$ is a NEP for the non-zero-sum game 
\cref{eq:Defnu1}/\cref{eq:Defnu2} if and only if for all $\tau,\sigma\in\mathcal{T}_0,$
	\begin{equation*}
	\pi^A_0\big(-R(\tau^*,\sigma^*)\big) \ge \pi^A_0\big(-R(\tau^*,\sigma)\big)\quad \text{and}\quad \pi^B_0\big(R(\tau^*,\sigma^*)\big) \ge \pi^B_0\big(R(\tau,\sigma^*)\big).
	\end{equation*}
\end{proposition}
\begin{proof}
    First note that by \cref{eq:IndPricePrimalCharact}, the indifference value for claim $H\in L^\infty$ and risk-aversion parameter $\alpha\in (0,\infty)$ satisfies 
    \begin{equation}\label{eq:ToRewriteIntermsofIndiffPrices}
    	\sup_{\vartheta\in\Theta} E\Big[-e^{-\alpha\big(C+H+\int_0^T\vartheta^{\mathrm{tr}}_sdS_s\big)}\Big] = e^{-\alpha \pi^{\alpha,C}_0(H)}\sup_{\vartheta\in\Theta} E\big[-e^{-\alpha\big(C+\int_0^T\vartheta^{\mathrm{tr}}_sdS_s\big)}\big]< 0.
    \end{equation}
Hence one obtains the required equivalence after substituting \cref{eq:ToRewriteIntermsofIndiffPrices} into \cref{eq:Defnu1,eq:Defnu2} for risk-aversion parameters $\alpha_A, \alpha_B$, exogenous endowments $C_A,C_B$,
and for the claims $R(\tau^*,\sigma^*)$, $R(\tau^*,\sigma)$, $R(\tau,\sigma^*)$.
\end{proof}

The game is typically of non-zero-sum type since under market incompleteness, the implication 
\begin{equation}\label{11.3.2017.1}
\pi^A_0(-H^1)\le \pi^A_0(-H^2)\ \implies \pi^B_0(H^1)\ge \pi^B_0(H^2)
\end{equation}
does not hold in general, for $H^1,H^2\in L^\infty$. On the other hand, in a complete market, the indifference valuations of both players are the replication cost 
which yields \cref{11.3.2017.1}.

Proposition~\ref{pro:NEPinTermsofIndiffPrice} establishes a relation to
optimal stopping problems under the utility-indifference valuation, which is not homogenous in the payoff. Before stating our contribution (Theorem~\ref{thm:NonlinSnellNEW}), let us discuss the existing literature on optimal stopping under nonlinear expectations.

\subsection{Literature review on nonlinear optimal stopping and new contribution}\label{14.6.2018.1}

Beyond the classical theory of optimal stopping under linear expectations surveyed in the seminal article El Karoui~\cite{ElKaroui1981},
the theory of optimal stopping under nonlinear expectation is in general also quite well-developed. Studies in the latter direction have mostly concentrated on sublinear expectations 
that are positively homogeneous, and hence can as well be associated to dynamic coherent risk measures; cf.\ among others \cite{KaratzasZamfirescu2005, Riedel2009, EkrenTouziZhang2014, NutzZhang2015}. 
But these do not include the European utility indifference valuation, which in general is neither subadditive nor positively homogeneous.  
The closest article to our work on the American indifference value is Bayraktar et al.~\cite{BayraktarKaratzasYao2010}, which solves the problem of optimal stopping under convex risk measures in a Brownian filtration setting. 
By making use of a representation of convex risky measures from Delbaen et al.~\cite{DelbaenPR10}, they consider risk measures that can be written as a worst case expectation of the payoff 
plus a proper convex penalty function in which the Girsanov kernels of equivalent probability measures are plugged. By this representation, which in turn relies on the predictable representation property of Brownian martingales as 
stochastic integrals, the problem can be solved by similar methods as for problems of robust (worst-case) combined stochastic control and optimal stopping, see, e.g., Karatzas and Zamfirescu~\cite{KaratzasZamfirescu2008}. 
In one aspect, our assumptions are slightly weaker than those of \cite{BayraktarKaratzasYao2010} since we only assume that the filtration is continuous instead of Brownian. 
But, more importantly, our methods are completely different from theirs. 
In \cite{leung.sircar.2009}, it is already stated that the optimal exercise time of an American claim
is given by the first time the nonlinear Snell envelope hits the payoff process. We think that Proposition~2.13 therein holds true, but we do not think that in its proof, 
Theorem~2.10 from Karatzas and Zamfirescu~\cite{KaratzasZamfirescu2005}, which deals with a best case optimal stopping problem that is positively homogeneous in the payoff, can be applied. 
On the other hand, Bayraktar and Yao~\cite{BayraktarYao2011-I, BayraktarYao2011-II} solve the optimal stopping problem for convex expectations by arguing with an up-crossing theorem 
for nonlinear expectations. Since the buyer's indifference value is concave in the random payoff, we cannot apply their results.  
Recently, Grigorova et al.~\cite{Grigorovaetal2017} have obtained results on optimal stopping under $g$-expectations for Lipschitz generators $g$ and for payoff processes 
only required to be optional (rather than c\`adl\`ag) with respect to a usual filtration generated by a Brownian motion and an independent Poisson random measure. 
We also cannot apply their results because the indifference valuation corresponds to a $g$-expectation with $g$ of quadratic growth.

We establish in \cref{thm:NonlinSnellNEW} the existence of a right-continuous ``Snell envelope'' corresponding to the American exponential 
utility indifference valuation. Crucial for the proof is both a continuity result from \cite{ManiaSchweizer05} (cf.\ property 7 in \cref{pro:IndiffValueProperties}) 
and its reverse, which we newly derive in the present article, cf.\ \cref{pro:ConvOfExpIMpliesConvInProb}.
The arguments in the proof of \cref{thm:NonlinSnellNEW} do not use an up-crossing theorem and hence could also be applied to more general nonlinear expectations. 
Recall that the classical up-crossing theorem states that every supermartingale (under a linear expectation) admits finite left and right limits over rationals 
(cf.\ e.g.\ Proposition~3.14 (i) in Chapter~1 of Karatzas and Shreve~\cite{KaratzasShreveBrownianMotion}). 
Our techniques are similar to the ones in the classical theory of optimal stopping under linear expectations with c\`adl\`ag payoff processes and filtrations satisfying the usual conditions 
(see e.g.\ Appendix D of Karatzas and Shreve~\cite{KaratzasShreveMathFi}). 
But, since we do not rely on the up-crossing theorem, that guarantees the existence of a right-limit 
process, we argue with a right-liminf process of the Snell envelope values at rational time points. It is not a priori clear that this process is right-continuous, so a critical part of 
our analysis is dedicated to verify this. First, we show that the process is progressively measurable and right continuous along stopping times. Then, an optional projection argument 
combined with the section theorem yields right-continuity up to evanescence. Our arguments rely on the continuity of the filtration at two places: first to identify 
(through the reverse continuity of the indifference operator, \cref{pro:ConvOfExpIMpliesConvInProb}) the optimal stopping time as the first time the payoff process meets the 
Snell envelope, and second to show that the defined right-liminf process is right-continuous along stopping times relying on the fact that every stopping time is predictable.

For supermartingales under nonlinear expectations, the up-crossing theorem holds if the (nonlinear) expectation of a sequence of nonnegative random variables  
explodes when the sequence tends pointwise to infinity on a set with positive probability (cf.\ hypothesis (H0) in \cite{BayraktarYao2011-I} and its use in the proof of the 
 nonlinear up-crossing Theorem 2.3 therein). 
For a European claim, the buyer's indifference value is a submartingale under the entropy minimizing martingale measure (EMMM) and thus the usual up-crossing theorem guarantees a c\`adl\`ag version 
(see Proposition 12 in \cite{ManiaSchweizer05} and also Theorem~3 in Bion-Nadal~\cite{Bion-Nadal2009}). 
For an American claim this is more delicate, as the price can decrease if the optimal execution time is missed.
Indeed, in general, the American indifference price is neither a sub- nor a supermartingale under the EMMM. 
In addition, although the American indifference price satisfies the supermartingale property with respect to the family of indifference valuation operators~$(\pi_t)_{t\in[0,T]}$,
the nonlinear expectation $\pi_0(\cdot)$ violates hypothesis (H0) of \cite{BayraktarYao2011-I} (see \cref{15.3.2017.1} for a 
counterexample), hence hindering a straightforward application of the (nonlinear) up-crossing theorem. 
However, though right continuity of the American indifference value is sufficient for our purpose in the proof of existence of a Nash equilibrium point for the GCC in \cref{thm:NEPforGamePb}, 
we complement \cref{thm:NonlinSnellNEW} by showing in \cref{rem:LeftLimitsSnellEnv} that the American indifference value is indeed c\`adl\`ag. 
This is achieved by writing the latter as a continuous function of the quotient of two supermartingales to which the up-crossing theorem for linear expectations can be applied. 

The following theorem is key in the construction of a NEP in \cref{thm:NEPforGamePb}. It provides existence and uniqueness of a right-continuous adapted process that dominates a given payoff process 
and firstly hits it at an optimal stopping time. In addition, this unique process only depends on the future payoff process
restricted to all events from which it is already known that they occur. 
These are the properties that are needed for dynamic programming, and it is natural to call the process the nonlinear Snell envelope with respect to  
the nonlinear expectation given by the European indifference valuation. The proof of the theorem is relegated to \cref{app:AppendixB}.
\begin{thm}[Snell envelope and optimal stopping]\label{thm:NonlinSnellNEW}
Let $\filt$ be continuous, $\alpha\in(0,\infty)$, $C\in L^\infty$, and let $L$ be a payoff process in $\mathcal{S}^\infty$. Then, there exists a right-continuous adapted process~$V$ with
\begin{equation}\label{7.1.2017.3}
V_t = \esssup_{\tau\in\mathcal{T}_t}\, \pi^{\alpha,C}_t(L_\tau),\ P\mbox{-a.s.\ for all\ }t\in[0,T]. 
\end{equation}
The process $V$ is unique up to evanescence (i.e., unique up to a global $P$-null set not depending on time) and possesses the following properties:
\begin{itemize}
\item[(i)] $\pi^{\alpha,C}_t(V_\tau)\le V_t$, $P$-a.s.\ for all
$t\in[0,T], \tau\in\mathcal{T}_t$.
\item[(ii)] If $L$ has no negative jumps, then 
$\hat{\tau}_t:=\inf\{u\ge t\ |\ V_u = L_u\}$ is a $[t,T]$-valued stopping time with 
$V_{\hat{\tau}_t}=L_{\hat{\tau}_t}$ $P$-a.s.\ and
$\pi^{\alpha,C}_0(L_{\hat{\tau}_t}) = \sup_{\tau\in\mathcal{T}_t}\, \pi^{\alpha,C}_0(L_\tau)$ for all $t\in[0,T]$. 
\item [(iii)] For two payoff processes $L^1, L^2$ with $L^1=L^2$ on 
$[\sigma,T]$, where $\sigma\in  \mathcal{T}_0$, one has  $V^1=V^2$ on $[\sigma,T]$ up to evanescence, for the associated processes $V^1, V^2$.
\item[(iv)] If $L=L_\sigma$ on $[\sigma,T]$, where $\sigma\in  \mathcal{T}_0$, then $V_\sigma=L_\sigma$, $P$-a.s.. 
\end{itemize}
\end{thm}

\begin{rem}
Like in the linear case, the American option price is in general not a martingale but only a supermartingale. It is an easy exercise to show that (i) holds with equality for all $t\in[0,T]$ and $\tau\in\mathcal{T}_t$ if and only if 
$L_t\le \pi^{\alpha,C}_t(L_T),\ P\mbox{-a.s. for all\ }t\in[0,T]$, i.e., $T$ is an optimal stopping time.
\end{rem}

\begin{rem}\label{rem:LeftLimitsSnellEnv}
In the proof of \cref{thm:NonlinSnellNEW}, the up-crossing theorem is not required. Thus, the arguments hold for quite general nonlinear expectations. However,
for the American indifference value, which is in general neither a super- nor a submartingale (w.r.t. $P$ or $Q^E$), we show in the following that the up-crossing theorem can nevertheless be used to prove that the dynamic value admits finite left and right limits over rationals. The indifference value can be written as
\begin{equation}\label{18.7.2017.1}
\begin{split}
\esssup_{\tau\in\mathcal{T}_t}\, \pi^{\alpha,C}_t(L_\tau) 
& = 
\esssup_{\tau\in\mathcal{T}_t}\, 
\left(
- \frac{1}{\alpha}\log\left(\frac{\displaystyle\esssup_{\vartheta\in\Theta} E_t\left[-e^{-\alpha\big(C+L_\tau+\int_t^T\vartheta^{\mathrm{tr}}_sdS_s\big)}\right]}{\displaystyle\esssup_{\vartheta\in\Theta} E_t\left[-e^{-\alpha\big(C+\int_t^T\vartheta^{\mathrm{tr}}_sdS_s\big)}\right]}\right)
\right)\\
& =   - \frac{1}{\alpha}\log\left(\frac{\displaystyle\esssup_{(\vartheta,\tau)\in\Theta\times\mathcal{T}_t} E_t\left[-e^{-\alpha\big(C+L_\tau+\int_t^T\vartheta^{\mathrm{tr}}_sdS_s\big)}\right]}{\displaystyle\esssup_{\vartheta\in\Theta} E_t\left[-e^{-\alpha\big(C+\int_t^T\vartheta^{\mathrm{tr}}_sdS_s\big)}\right]}\right).
\end{split}
\end{equation}
Let us show that $A_t:={\rm esssup}_{(\vartheta,\tau)\in\Theta\times\mathcal{T}_t} E_t\Big[-e^{-\alpha\big(C+L_\tau+\int_t^T\vartheta^{\mathrm{tr}}_sdS_s\big)}\Big]$ satisfies the 
supermartingale property
$A_t\ge E_t[A_{t+h}]$,\ $P$-a.s.\ for all $t, t+h\in [0,T]$. 
Indeed, due to the choice of $\Theta$ in \cref{eq:DefStrategiesTheta} the set 
$\left\{E_{t+h}\Big[-e^{-\alpha\big(C+L_\tau+\int_{t+h}^T\vartheta^{\mathrm{tr}}_sdS_s\big)}\Big]\ \Big\lvert\ (\vartheta,\tau)\in\Theta\times\mathcal{T}_{t+h}\right\}$ 
is maximum-stable and thus
\begin{align*}
E_t[A_{t+h}] & =  
E_t\left[ \esssup_{(\vartheta,\tau)\in\Theta\times\mathcal{T}_{t+h}}
E_{t+h}\left[-e^{-\alpha\big(C+L_\tau+\int_{t+h}^T\vartheta^{\mathrm{tr}}_sdS_s\big)}\right]\right]\\
& =
\esssup_{(\vartheta,\tau)\in\Theta\times\mathcal{T}_{t+h}}
E_t\left[  E_{t+h}\left[-e^{-\alpha\big(C+L_\tau+\int_{t+h}^T\vartheta^{\mathrm{tr}}_sdS_s\big)}\right]\right]
\le A_t,\quad P\mbox{-a.s.}.
\end{align*}
Since the denominator in the last line of \cref{18.7.2017.1} coincides with $A_t$ for $L=0$, it also satisfies the supermartingale property.  
We conclude that there exists an event with full probability on which 
 for all $t\in\R_+$ the limits 
\begin{equation}\label{18.7.2017.2}
\lim_{\substack{s\to t\\ s<t,\, s\in\Q}} \esssup_{\tau\in\mathcal{T}_s}\, \pi^{\alpha,C}_s(L_\tau)
\quad\mbox{and}\quad
\lim_{\substack{s\to t\\ s>t,\, s\in\Q}} \esssup_{\tau\in\mathcal{T}_s}\, \pi^{\alpha,C}_s(L_\tau) \text{ exist and are finite.}
\end{equation}
This is because by applying the up-crossing theorem (see e.g.\ Proposition~3.14 (i) in Chapter~1 of \cite{KaratzasShreveBrownianMotion}) we have this property for both the 
numerator and the denominator in the last line of \cref{18.7.2017.1}, and the quotient being bigger than $\exp(-\alpha||L||_{\infty})$ is bounded away from zero. 
By \cref{18.7.2017.2}, it is immediate that the process \cref{12.1.2017.1}
in  the proof of \cref{thm:NonlinSnellNEW} possesses finite left limits
up to evanescence. Together with \cref{22.12.2016.1}, \cref{18.7.2017.2} also implies that \cref{12.1.2017.1} is right-continuous 
up to evanescence. Thus, we even have that the Snell envelope process $V$ of \cref{thm:NonlinSnellNEW} is c\`adl\`ag.
\end{rem}

\section{Examples}\label{14.6.2018.2}

This section provides some economic intuition behind equilibria 
in incomplete markets. The phenomena described in the following two examples cannot occur in complete markets, in which both players stop the contract with the aim to maximize their expected payoffs under the unique martingale measures, regardless of their risk aversions and random endowments. 
For simplicity, we consider the case that $S=1$, i.e., trading gains in the underlyings need not be considered. The payoff processes are not bounded, but 
the examples satisfy the assumptions in Hamad\`ene and Zhang~\cite{HamadeneZhang2010},
in which endowments can be considered by modifying the payoff processes in a straight forward way.

The first example, in which equilibria are unique, 
illustrates how the writer's optimal stopping time depends on her random endowment. 
\begin{example}[Impact of endowments on optimal stopping times]\label{9.5.2018.1}
Let $W$ be a standard Brownian motion w.r.t. the filtration~$\mathbb{F}$. Consider the GCC with payoff processes $X_t = W_t + \mu t$ and $Y_t = W_t + \mu t + \delta$, $\ t\in[0,T]$,\ $\mu,\delta\in\R_+$. The process $X$ can be interpreted as the value of a nontraded asset and $\delta$ as the penalty the seller has to pay if she recalls the GCC prematurely.
Assume that $\alpha_B/2 < \mu < \alpha_A /2$, 
$0<\delta < (\alpha_A /  2 +  \mu) T$, and $C_B=0$. Since for any $\sigma\in\mathcal{T}_0$, the process 
\beao
t\mapsto & & \ -\exp\left(-\alpha_B R(t,\sigma)\right)\\
& & = -\exp\left(-\alpha_B W_{t\wedge\sigma} - \frac{\alpha^2_B}{2} (t\wedge\sigma)\right)
\exp\left(\alpha_B\left(\frac{\alpha_B}{2}-\mu\right)(t\wedge\sigma) -\alpha_B \delta 1_{(t>\sigma)}\right)
\eeao 
is an (optional) submartingale, a dominant strategy for the buyer is to stop at maturity~$T$. This means that the drift is high enough to compensate the buyer for the inventory risk of the GCC. We now distinguish two cases for the endowment of the seller: 

Case 1: $C_A=0$. In this case, the seller has to solve the optimal stopping problem
\beam\label{2.4.2018.1}
\sup_{\sigma\in\mathcal{T}_0} E\left[-\exp(\alpha_A R(T,\sigma))\right].
\eeam
Applying the change of measure $d\wt{P}/dP=\exp(\alpha_A W_T- \alpha_A^2 T/2)$ yields
\beao
& & E\left[-\exp(\alpha_A R(T,\sigma))\right]\\
& & = E\left[-\exp\left(\alpha_A W_{\sigma} - \frac{\alpha^2_A \sigma}{2}\right)
\exp\left(\alpha_A\left(\frac{\alpha_A}{2}+\mu\right)\sigma +\alpha_A \delta 1_{(\sigma<T)}\right)\right]\\
& & = -E^{\wt{P}}\left[\exp\left(\alpha_A\left(\frac{\alpha_A}{2}+\mu\right)\sigma +\alpha_A \delta 1_{(\sigma<T)}\right)\right],
\eeao 
and the pathwise minimizer of the expression under the expectation $E^{\wt{P}}$ is given by $\sigma=0$. Consequently, $\sigma \equiv 0$ also solves (\ref{2.4.2018.1}), i.e., the seller recalls the contract 
immediately.

Case 2: $C_A=W_T+\mu T$, i.e., the seller holds a long position in the 
nontraded asset. Thus, she has to solve the problem
\beao
\sup_{\sigma\in\mathcal{T}_0} E\left[-\exp(-\alpha_A(W_T+\mu T-R(T,\sigma)))\right].
\eeao
After the change of measure $d\wt{P}/dP=\exp(-\alpha_A W_T-\alpha_A^2 T/2)$, the problem reads
\beao
c\,\sup_{\sigma\in\mathcal{T}_0} E^{\wt{P}}\left[-\exp(\alpha_A R(T,\sigma))\right]
=c\, \sup_{\sigma\in\mathcal{T}_0} E^{\wt{P}}\left[-\exp(\alpha_A(\wt{W}_\sigma-\alpha_A\sigma +\mu\sigma + \delta 1_{(\sigma<T)}))\right],
\eeao
where $\wt{W}_t := W_t + \alpha_A t,\ t\in[0,T],$ is a $\wt{P}$-standard Brownian motion by Girsanov's theorem and $c:=\exp(\alpha_A(\alpha_A/2 - \mu)T)$. By $\alpha_A (\mu -\alpha_A/2) <0$, the process
\beao
t\mapsto & & \ -\exp\left(\alpha_A \wt{W}_t - \alpha^2_A t +\alpha_A \mu t +\alpha_A \delta 1_{(t<T)}\right)
\eeao 
is a submartingale. This yields that the supremum is attained at $\sigma\equiv T$, i.e., the contract is settled at maturity. 

Summing up, without endowment, the seller recalls the claim immediately to reduce her risk. By contrast, if she holds a long position in the nontraded asset, a short position in the GCC is a perfect hedging instrument.  Thus, she tolerates the positive drift of the underlying and holds her short position in the GCC up to maturity. 
\end{example}
\begin{rem}\label{5.9.2018.1}
It is shown by Anthropelos and \v{Z}itkovi\'c~\cite{anthropelos.zitkovic} 
(see Remark~3.17 and Lemma~A.7 therein) that there is an,  up to replicable payoffs unique, ``mutually agreeable'' European claim given by  
$B^\star:=(\alpha_A C_A - \alpha_B C_B)/(\alpha_A+\alpha_B)$, which the buyer~$B$ purchases from the seller~$A$ and which leads to a Pareto-optimal allocation. An immediate consequence is that there is a mutual incentive of the agents to trade European claims if and only if $\alpha_A C_A - \alpha_B C_B$ is nonreplicable in the underlyings. 
After trading $B^\star$, the endowments of the buyer and the seller are given by $\alpha_A(C_A + C_B)/(\alpha_A+\alpha_B)$ and $\alpha_B(C_A + C_B)/(\alpha_A+\alpha_B)$, respectively, up to replicable payoffs.
The special feature of exponential utility is that by the constant absolute risk-aversion, the risky-sharing
does not depend on the distribution of the aggregate endowment~$C_A+C_B$ between the agents before trading. Consequently, the above mentioned Pareto-optimal allocation can be reached by a mutually agreeable  trade for {\em any} initial distribution of $C_A+C_B$ between the agents. 
%
%

Applied to Example~\ref{9.5.2018.1}, this yields that in Case~2, the agents would trade European claims and in Case~1 not.  This is reflected in the equilibria of the stopping game, in which the contract is canceled immediately in Case~1 and 
not before maturity in Case~2.
%
%
In Case~2 of Example~\ref{9.5.2018.1}, the Pareto-optimal claim purchased by the buyer is given by $\alpha_A W_T/(\alpha_A+\alpha_B)$ up to constants.
Of course, it is not surprising that the equilibrium payoff of the GCC does in general not reproduce the optimal claim in the model of \cite{anthropelos.zitkovic} in which the agents can trade arbitrary European claims.

%
%
Another consequence of \cite{anthropelos.zitkovic} is that if $\alpha_A C_A - \alpha_B C_B$ is replicable, then it is impossible that both agents profit from entering into the GCC contract. However, by $X\le Y$, i.e., because premature stopping is penalized, this does not necessarily imply that an existing contract would be canceled immediately by one of the players. In addition, it is obvious that the optimal stopping game does, in general, not simplify as in the special case of a complete financial market, in which both players maximize their expected payoffs under the unique martingale measure.
\end{rem}
The next example shows that the equilibrium values of the game are in general not unique. 
\begin{example}[Equilibrium values are not unique]\label{5.9.2018.2}
Let $W$ be again a standard Brownian motion.
Consider the payoff processes $X_t=W_t$ and $Y_t=W_t + \delta$ with $0<\delta<(\alpha_A/2)T$, $\alpha_B>0$, and $C_A=C_B=0$. Now, both 
risk-averse players have an incentive to stop as early as possible, but 
by the penalty, they would prefer that the other player stops first.
%
%
If the other player stops at $T$, the best response of both players is to stop at zero. This implies that both pairs
\beam\label{12.5.2018.3}
(0,T)\quad\mbox{and}\quad(T,0)\quad\mbox{are NEPs},
\eeam 
but with different values $(u_A,u_B)$. 
The iterated best response in (\ref{eq:ApproxNEPOdd})/(\ref{3.6.2018.1})
leads to the equilibrium~$(0,T)$. By symmetry, it is also possible to start with the best response of the seller which leads to $(T,0)$.

But, the game possesses also other NEPs with $P(\tau\wedge\sigma>0)=1$ that are not the outcome of the above mentioned iterations starting with $(T,T)$. 
The equilibria are based on a {\em partial} payment of the penalty~$\delta$. 
The idea is that the players ``toss a coin'' who has to stop first. But, since the 
filtration $\mathbb{F}$ is continuous, a simple coin toss at time zero 
cannot be embedded in the model. Thus, in the following, we approximate such a 
behavior, which leads to nontrivial NEPs. 
%
%

Let $\eps>0$ be small enough s.t. 
\beam\label{9.5.2018.2}
\delta<\frac{\alpha_A}{2}(T-\eps),
\eeam
\beam\label{12.5.2018.1}
\exp\left(\frac{\alpha^2_A \eps}{2}\right)
\le \frac{2\exp(\alpha_A\delta)}{1+\exp(\alpha_A\delta)},\quad
\mbox{and}\quad
\exp\left(\frac{\alpha^2_B \eps}{2}\right)
\le \frac2{1+\exp(-\alpha_B\delta)}.
\eeam  
Consider at first the pair 
\beam\label{9.5.2018.3}
\tau_0:=\eps 1_{\{W_\eps\ge 0\}}+T 1_{\{W_\eps < 0\}}\quad\mbox{and}\quad \sigma_0:=\eps 1_{\{W_\eps < 0\}} + T 1_{\{W_\eps\ge 0\}}.
\eeam
This means that the random variable~$W_\eps$ is used for the coin toss which decides who has to stop at $\eps$. The pair~$(\tau_0,\sigma_0)$ is not yet a NEP. Namely, by the tail probabilities of the normal distribution, it happens that a player improves her expected utility by stopping even before $\eps$ if the conditional probability that the other player stops is sufficiently small compared to the linear expected loss she suffers by the waiting time. 
Now, we apply the iterated best response from (\ref{eq:ApproxNEPOdd})/(\ref{3.6.2018.1}) to the stopping game restricted to the interval~$[0,\eps]$ with terminal payoff $W_\eps + \delta 1_{\{W_\eps < 0\}}$. It follows directly from the proof of Theorem~\ref{thm:NEPforGamePb} that the resulting limiting pair, denoted by 
$(\wt{\tau}^\star,\wt{\sigma}^\star)$, is a NEP of the modified game. Based on this pair, we define the $[0,T]$-valued stopping times
\beao
\tau^\star:=\wt{\tau}^\star 1_{(\wt{\tau}^\star<\eps)} + \eps 1_{(\wt{\tau}^\star=\eps,\ W_\eps\ge 0)} + T 1_{(\wt{\tau}^\star=\eps,\ W_\eps<0)}
\eeao
and
\beao
\sigma^\star:=\wt{\sigma}^\star 1_{(\wt{\sigma}^\star<\eps)} + \eps 1_{(\wt{\sigma}^\star=\eps,\ W_\eps< 0)} + T 1_{(\wt{\sigma}^\star=\eps,\ W_\eps\ge 0)}.
\eeao
First observe that $(\wt{\tau}^\star,\wt{\sigma}^\star)$ and $(\tau^\star,\sigma^\star)$ lead to the same payoffs in their respective games.  
Since in addition, by (\ref{9.5.2018.2}), $(\tau_0,\sigma_0)$ is a NEP of the corresponding game started at time~$\eps$, it follows that $(\tau^\star,\sigma^\star)$ is a NEP in the original game.
%
%

In addition, it can easily be seen that $P(\tau^\star=0)=P(\sigma^\star=0)=0$. 
Indeed, for the buyer, one obtains the estimate
\beam\label{11.5.2018.1}
u_B(\tau_0,\sigma^\star) & = & - E\left[\exp(-\alpha_B W_{\sigma^\star\wedge\eps} - \alpha_B \delta 1_{\{\sigma^\star< \tau_0\}})\right]\nonumber\\
& \ge & - E\left[\exp(-\alpha_B W_{\eps} - \alpha_B \delta 1_{\{\sigma^\star<\tau_0\}})\right]\nonumber\\
& \ge & - E\left[\exp(-\alpha_B W_{\eps} - \alpha_B \delta 1_{\{W_\eps< 0\}})\right]\nonumber\\
& = & - E\left[\exp(-\alpha_B W_{\eps})\left(1+(\exp(-\alpha_B \delta)-1)1_{\{W_\eps<0\}}\right)\right]\nonumber\\
& \ge & -E\left[\exp(-\alpha_B W_{\eps})\right]\left(1+(\exp(-\alpha_B \delta)-1)P(W_\eps<0)\right)\nonumber\\
& = & -\frac12 E\left[\exp\left(\frac{\alpha^2_B \eps}{2}\right)\right]\left(1+(\exp(-\alpha_B\delta)\right)\nonumber\\
& > & -1 = u_B(0,\sigma^\star).
\eeam
Here, the first inequality follows from the submartingale property of $\exp(-\alpha_B W)$ and the second inequality from $\{W_\eps< 0\}\subset \{\sigma^\star<\tau_0\}$.
The third inequality can be derived from Girsanov's theorem applied to the measure
$d\wt{P}/dP=\exp(-\alpha_B W_\eps-\alpha_B^2 \eps/2)$. The strict inequality holds by the choice of $\eps$ in (\ref{12.5.2018.1}). 
(\ref{11.5.2018.1}) implies that $\tau\equiv 0$ cannot be the best response to $\sigma^\star$ and thus $P(\tau^\star>0)>0$. If $\mathcal{F}_0$ is $P$-trivial, we are already done. Otherwise, the conditional version of (\ref{11.5.2018.1}) yields $P(\tau^\star=0)=0$.
Then, the same calculations yield $P(\sigma^\star=0)=0$. Here, we  
use that $u_A(\tau^\star,0) = -\exp(\alpha_A\delta)$ by $P(\tau^\star=0)=0$.
 
This means that
without any information on $W_\eps$, it cannot be optimal to stop at zero and pay the full penalty (or resign to get payed it).\\

On the other hand, it follows from Lemma~A.7 of Anthropelos and \v{Z}itkovi\'c~\cite{anthropelos.zitkovic}  that (\ref{12.5.2018.3}) are the only Pareto-optimal NEPs. Indeed, for stopping times $\tau,\sigma$, the payoff~$R(\tau,\sigma)=W_{\tau\wedge\sigma} + \delta 1_{\{\sigma<\tau\}}$ is deterministic if and only if $P(\tau=0)=1$ or $P(\sigma=0,\ \tau>0)=1$. But for a non-deterministic $R(\tau,\sigma)$,
the lemma tells that $\pi^B(R(\tau,\sigma)) + \pi^A(-R(\tau,\sigma))<0$. 
This means that the players may agree that the seller pays to the buyer the 
amount $\left[ \pi^B(R(\tau,\sigma)) - \pi^A(-R(\tau,\sigma)\right]/2$ as a compensation for exercising her claim. This would improve the expected utility of both players
compared to playing the game with $(\tau,\sigma)$.
\end{example}

\section{Proof of Theorem~\ref{thm:NEPforGamePb}}\label{sect:proofMainResult}
For $\pi^A$ and $\pi^B$ from \cref{26.7.2017.1},
define the functionals $J_A,J_B:\mathcal{T}_0\times \mathcal{T}_0\longrightarrow \R$ by 
\begin{equation*}
	J_B(\tau,\sigma) := \pi^B_0\big(R(\tau,\sigma)\big)\quad \text{and}\quad J_A(\tau,\sigma) := \pi^A_0\big(-R(\tau,\sigma)\big).
\end{equation*}
By \cref{pro:NEPinTermsofIndiffPrice}, a NEP of the game 
\cref{eq:Defnu1}/\cref{eq:Defnu2} is a pair $(\tau^*,\sigma^*)\in\mathcal{T}_0^2$ satisfying 
\begin{equation}\label{eq:NEPdefnJGeneral}
	J_B(\tau^*,\sigma^*) \ge J_B(\tau,\sigma^*)\quad \text{and}\quad J_A(\tau^*,\sigma^*) \ge J_A(\tau^*,\sigma),\quad \text{for all } \tau,\sigma\in\mathcal{T}_0. 
	\end{equation}
To prove the existence of a pair~$(\tau^*,\sigma^*)$ satisfying \cref{eq:NEPdefnJGeneral}, we follow the ideas of Hamad\`ene and 
Zhang~\cite{HamadeneZhang2010}.
But, to deal with the nonlinearity of the indifference valuation, we have to adjust the proof for linear expectations at various places, mainly by
applying \cref{pro:ConvOfExpIMpliesConvInProb,thm:NonlinSnellNEW}.
To receive readability, we repeat the main proof of  \cite{HamadeneZhang2010} while omitting only the parts which are one-to-one translations.

One first constructs best response strategies through two sequences of stopping 
times $(\tau_{2n+1})_{n\in\N_0}$, $(\tau_{2n+2})_{n\in\N_0}$ for the buyer and the seller, respectively, and shows that 
both sequences are nonincreasing. Finally, one shows that 
the limits $\tau^*_1$ of $\tau_{2n+1}$ and $\tau^*_2$ of $\tau_{2n+2}$ as $n$ tends to infinity define a NEP $(\tau^*_1,\tau^*_2)=(\tau^*,\sigma^*)$ satisfying \cref{eq:NEPdefnJGeneral}. Of course, the monotonicity of the sequences of stopping times is key, since otherwise the best response strategies may oscillate. 

Let $\tau_1:=\tau_2:=T$. Given $\tau_{2n-1},\tau_{2n}$ for some $n\in\N$, we want to construct
$\tau_{2n+1}$ as follows:
Consider the payoff process $L^{2n+1}\in\mathcal{S}^\infty$ defined for $t\in[0,T]$ by 
\begin{equation}\label{eq:PayoffProcessesOddEven}
L^{2n+1}_t := X_t\mathds{1}_{\{t<\tau_{2n}\}}+\big(X_T\mathds{1}_{\{\tau_{2n}=T\}}+Y_{\tau_{2n}}\mathds{1}_{\{\tau_{2n}<T\}}\big)\mathds{1}_{\{t\ge\tau_{2n}\}}. 
\end{equation}
Under \cref{eq:aspFiltContinuous,eq:XleY,eq:NoNegJumps}, \cref{thm:NonlinSnellNEW} can be applied to $L^{2n+1}$, and an optimal stopping time of 
$\sup_{\tau\in\mathcal{T}_0}\, \pi^B_0(L^{2n+1}_\tau)$ is given by 
\begin{equation}\label{eq:OptStoppingTimesOdd}
	\tilde{\tau}_{2n+1}:=\inf\{t\ge 0: V^{2n+1}_t=L^{2n+1}_t\} = \inf\{t\ge 0:V^{2n+1}_t=X_t\}\wedge \tau_{2n},
\end{equation}
where the $\pi^B$-Snell envelope~$V^{2n+1}$ is the unique right-continuous adapted process satisfying
\begin{equation}\label{eq:SnellEnvOdd}
	V^{2n+1}_t = \esssup_{\tau\in\mathcal{T}_t}\, \pi^B_t(L^{2n+1}_\tau),\quad P\mbox{-a.s.}\ t\in[0,T] 
\end{equation}
(the second equality in \cref{eq:OptStoppingTimesOdd} follows from  \cref{thm:NonlinSnellNEW}(iv)). 

Furthermore, we define  
\begin{equation}\label{eq:ApproxNEPOdd}
	\tau_{2n+1}:= \tilde{\tau}_{2n+1}\mathds{1}_{\{\tilde{\tau}_{2n+1}<\tau_{2n}\}} + \tau_{2n-1}\mathds{1}_{\{\tilde{\tau}_{2n+1}=\tau_{2n}\}}.
\end{equation}
\begin{rem}
The payoff process~$L^{2n+1}$ is chosen to be c\`adl\`ag. This comes at the price that $L^{2n+1}_\tau$ differs from $R(\tau,\tau_{2n})$ on the set~$\{\tau=\tau_{2n}<T\}$.
Thus, it is not yet clear that $\tilde{\tau}_{2n+1}$ is a best response strategy to $\tau_{2n}$. In addition, one takes $\tau_{2n+1}$ instead of
$\wt{\tau}_{2n+1}$. Yet, it is not even clear that $\tau_{2n+1}$ is a stopping time, and, a fortiori, that it is also a best response  strategy to $\tau_{2n}$.
\end{rem}
Given $\tau_{2n},\tau_{2n+1}$, the response~$\tau_{2n+2}$ of the seller~$A$ is defined in the same way by
\begin{equation}\label{3.6.2018.1}
	\tau_{2n+2}:= \tilde{\tau}_{2n+2}\mathds{1}_{\{\tilde{\tau}_{2n+2}<\tau_{2n+1}\}} + \tau_{2n}\mathds{1}_{\{\tilde{\tau}_{2n+2}=\tau_{2n+1}\}},
\end{equation}
where	$\tilde{\tau}_{2n+2}:=\inf\{t\ge 0:V^{2n+2}_t=L^{2n+2}_t\} = \inf\{t\ge 0:V^{2n+2}_t=-Y_t\}\wedge \tau_{2n+1}$
with payoff process
\begin{equation*}
L^{2n+2}_t:=-X_{\tau_{2n+1}}\mathds{1}_{\{t\ge\tau_{2n+1}\}} -Y_t\mathds{1}_{\{t<\tau_{2n+1}\}}
\end{equation*}
and $\pi^A$-Snell envelope $V^{2n+2}$ satisfying
\begin{equation*}
	V^{2n+2}_t = \esssup_{\tau\in\mathcal{T}_t}\, \pi^A_t(L^{2n+2}_\tau),\quad P\text{-a.s.}\ t\in[0,T].
\end{equation*}
\begin{lemma}\label{lem:PropSTimes} Assume \cref{eq:aspFiltContinuous,eq:XleY,eq:NoNegJumps}. Then 
	\begin{enumerate}
		\item For any $n\in\N$, $\tau_n$ is a stopping time and $\tau_{n+2}\le \tau_n$.
		\item On the event $\{\tau_{n+1}=\tau_n\}$, $n\in\N$, one has $\tau_m=T$ for all $m\le n$.
		\item For any $n\in\N$, $\{\tau_n<\tau_{n+1}\}\subset \{\tilde{\tau}_{n+2}\le \tau_n\}$ .
	\end{enumerate}
\end{lemma}
\begin{proof}
The assertions are the same as in Lemmas~3.1 and 3.2 in \cite{HamadeneZhang2010}, and the proof is  one-to-one. Only in (3.8) of \cite{HamadeneZhang2010}, the linear
Snell envelope has to be replaced by the nonlinear Snell envelope from \cref{thm:NonlinSnellNEW}. Here, we need properties~(iii) and, again, (ii)  
of \cref{thm:NonlinSnellNEW}.
\end{proof}
The following lemma shows that $\tau_{2n+1},\tau_{2n+2}$ are indeed best
responses. Though its proof is analogous to that of Lemma~3.3 in \cite{HamadeneZhang2010}, we provide it in \cref{app:AppendixB} for the reader's convenience since adjustments are required at many places.
\begin{lemma}\label{lem:IneqJitauAndtaun}
 Assume \cref{eq:aspFiltContinuous,eq:XleY,eq:NoNegJumps}. Then for any $\tau\in \mathcal{T}_0$ and $n\in\N$, it holds that
 \begin{equation}\label{eq:ApproximateNEP}
 	J_B(\tau,\tau_{2n})\le J_B(\tau_{2n+1},\tau_{2n})\quad \text{and}\quad J_A(\tau_{2n+1},\tau)\le J_A(\tau_{2n+1},\tau_{2n+2}).
 \end{equation}
\end{lemma}
By the monotonicity from \cref{lem:PropSTimes}(i), 
the best response strategies possess pointwise limits
\[
\tau^*_1:=\lim_{n\to\infty}\tau_{2n+1}\quad\mbox{and}\quad \tau^*_2:=\lim_{n\to\infty} \tau_{2n}
\]
that are of course again stopping times. To prove that $(\tau^*_1,\tau^*_2)$ is a NEP,
it only remains to show that the operations of taking limits and applying the indifference value operator can be interchanged.
The latter is done in the following two lemmas, which are proven in \cref{app:AppendixB}, and used
in  the proof of \cref{thm:NEPforGamePb} that we provide directly thereafter.
\begin{lemma}\label{lem:1stEqTowardsNEP}
	Assume \cref{eq:aspFiltContinuous,eq:XleY,eq:NoNegJumps}. Then for any $\tau\in\mathcal{T}_0$, it holds that
	\begin{enumerate}
		\item $J_B(\tau,\tau_{2n})\longrightarrow J_B(\tau,\tau^*_2)$, as $n\to\infty$.
		\item $P[\tau=\tau^*_1<T]=0$ implies $J_A(\tau_{2n+1},\tau)\longrightarrow J_A(\tau^*_1,\tau)$, as $n\to\infty$.
	\end{enumerate}
\end{lemma}

\begin{lemma}\label{lem:2ndEqTowardsNEP}
	Assume \cref{eq:aspFiltContinuous,eq:XleY,eq:NoNegJumps}. Then, 
	\begin{enumerate}
		\item $J_B(\tau_{2n+1},\tau_{2n})\longrightarrow J_B(\tau^*_1,\tau^*_2)$, as $n\to\infty$.
		\item $J_A(\tau_{2n+1},\tau_{2n+2})\longrightarrow J_A(\tau^*_1,\tau^*_2)$, as $n\to\infty$.
	\end{enumerate}
\end{lemma}
We are now ready to give a proof to \cref{thm:NEPforGamePb}.
\begin{proof}[Proof of \cref{thm:NEPforGamePb}]
	By \cref{lem:IneqJitauAndtaun,lem:1stEqTowardsNEP,lem:2ndEqTowardsNEP} we have 
	$	J_B(\tau,\tau^*_2)\le J_B(\tau^*_1,\tau^*_2)$, for all $\tau\in\mathcal{T}_0$
and 
\begin{equation}\label{eq:ToObtainNEPfinal}
		J_A(\tau^*_1,\tau)\le J_A(\tau^*_1,\tau^*_2),\ \text{ for all } \tau\in\mathcal{T}_0\text{ satisfying } P[\tau=\tau^*_1<T]=0.
	\end{equation}
To obtain that $(\tau^*_1,\tau^*_2)$ indeed is a NEP, it remains to show that \cref{eq:ToObtainNEPfinal} holds for arbitrary $\tau\in\mathcal{T}_0$. To this end, let $\tau\in\mathcal{T}_0$
and define the sequence $(\hat{\tau}^n)_{n\in\N}\subset \mathcal{T}_0$ by 
$$
\hat{\tau}^n:= \left((\tau+n^{-1})\wedge T\right)\mathds{1}_{\{\tau=\tau^*_1<T\}}+ \tau\mathds{1}_{\Omega\setminus \{\tau=\tau^*_1<T\}},\quad n\in\N.
$$
Then $\{\hat{\tau}^n=\tau^*_1<T\}=\emptyset$ for any $n\in\N$ so that \cref{eq:ToObtainNEPfinal} gives
\[
	J_A(\tau^*_1,\hat{\tau}^n)\le J_A(\tau^*_1,\tau^*_2)\quad \text{for all } n\in\N.
\]
Furthermore $\hat{\tau}^n\downarrow \tau$ almost surely as $n\uparrow \infty$. With the right-continuity of the bounded process $t\mapsto R(\tau^*_1,t)$
and the continuity of $\pi^A_0(\cdot)$, this implies 
$J_A(\tau^*_1,\tau)\le J_A(\tau^*_1,\tau^*_2)$. Overall, we have shown that $(\tau^*_1,\tau^*_2)\in \mathcal{T}_0^2$ is a NEP and the proof is completed.
\end{proof}

\section{Appendix}\label{app:AppendixB} 
This section contains the proofs of \cref{pro:ConvOfExpIMpliesConvInProb,thm:NonlinSnellNEW,lem:IneqJitauAndtaun,lem:1stEqTowardsNEP,lem:2ndEqTowardsNEP}.

\begin{proof}[Proof of \cref{pro:ConvOfExpIMpliesConvInProb}]
Without loss of generality let $C=0$. By Theorem~13 in Mania and Schweizer~\cite{ManiaSchweizer05} and following the proof of Proposition~14 
therein, we know that for any $n\in\N$
\begin{equation}\label{20.12.2016.2}
	\pi^{\alpha,0}_0(H^n+\eta^n)- \pi^{\alpha,0}_0(H^n) = E^{Q^n}[\eta^n],
\end{equation}
for $Q^n\sim P$ given by $dQ^n = \mathcal{E}\left(\frac{\alpha}{2}(L^n(\alpha) + \wt{L}^n(\alpha)\right)_TdQ^E =: Z^n_TdQ^E$, with $L^n(\alpha)$ and $\wt{L}^n(\alpha)$ being $BMO\left(Q^E\right)$-martingales satisfying 
\begin{equation}\label{eq:BMOboundLalpha}
\sup_{n\in\N}\left\lVert \frac{\alpha}2\left(L^n(\alpha)+\wt{L}^n(\alpha)\right)\right\rVert_{BMO(Q^E)}\le \frac{\alpha}2 \sup_{n\in\N}\left(e^{\lVert H_n+\eta_n\rVert_\infty}+e^{\lVert H_n\rVert_\infty}\right)^2<\infty,
\end{equation}
where $Q^E$ denotes the EMMM, and we refer to Kazamaki~\cite{Kazamaki94} for some essentials on BMO theory. In the proof of Theorem~2.4 in \cite{Kazamaki94}, 
implication~$(a)\implies(b)$, the parameter~$p>1$ only has to satisfy $||M||_{BMO}<\sqrt{2}(\sqrt{p}-1)$. Thus, applied to 
$M:=\frac{\alpha}2\left(L^n(\alpha)+\wt{L}^n(\alpha)\right)$, and using
\cref{eq:BMOboundLalpha}, $p$ can be chosen uniformly in $n$ and one gets
\begin{equation*}
\sup_{n\in\N}E^{Q^E}\left[\left(Z^n_T\right)^{-\frac{1}{p-1}}\right]\le c_p,
\end{equation*}
where $c_p>0$ is a universal constant.
Then, H\"older's inequality gives for all $n\in\N$,
\begin{equation}\label{eq:reverseHolderClassical}
\left(E^{Q^E}\left[|\eta_n|^{\frac{1}{p}}\right]\right)^p
\le E^{Q^E}\left[|\eta_n|Z^n_T\right]
\left(E^{Q^E}\left[\left(Z^n_T\right)^{-\frac{1}{p-1}}\right]\right)^{p-1}
\le E^{Q^n}[-\eta_n]c_p^{p-1}.
\end{equation}
Since the LHS of \cref{20.12.2016.2,} tends to $0$, \cref{eq:reverseHolderClassical} implies 
that $\eta_n$ converges to $0$ in $Q^E$-probability.
Because $Q^E\sim P$ holds, the assertion follows.
\end{proof}

\begin{proof}[Proof of \cref{thm:NonlinSnellNEW}]
For notational simplicity, we denote $\pi:=\pi^{\alpha,C}$. For all $s\in\Q\cap [0,T]$, fix throughout the proof a version $V_s:=\esssup_{\tau\in\mathcal{T}_s}\, \pi_s(L_\tau)$ 
satisfying $V_s\ge L_s$ and for technical convenience set $V_t:=L_T$ for $t>T$. Define 
\begin{equation}\label{12.1.2017.1}
\wt{V}_t:=\liminf_{s\ge t, s\in\Q, s\to t}V_s:=\sup_{m\in\N}\inf_{s\in[t,t+1/m]\cap\Q}V_s,\quad t\in[0,T]. 
\end{equation}
By right-continuity of $L$, we have $\wt{V}\ge L$. Furthermore let $t_0\in[0,T]$. The real-valued mapping $\wt{V}|_{\Omega\times[0,t_0]}$ can
be written as
\begin{equation*}
\begin{split}\wt{V}_t(\omega)
& = \mathds{1}_{\{t_0\}}(t) \wt{V}_{t_0}(\omega) + \mathds{1}_{(t<t_0)}\sup_{m\in\N}\inf_{s\in[t,(t+\frac{1}{m})\wedge t_0]\cap\Q}V_s(\omega)\\
& = \mathds{1}_{\{t_0\}}(t) \wt{V}_{t_0}(\omega) + \mathds{1}_{(t<t_0)}\sup_{m\in\N}\inf_{\substack{s\in\Q,\\ s\le t_0}}\left( \mathds{1}_{[s-\frac{1}{m},s]}(t)V_s(\omega) + \infty \mathds{1}_{[s-\frac{1}{m},s]^{c}}(t)\right).
\end{split}
\end{equation*}
By the usual conditions, $\wt{V}_{t_0}$ is $\mathcal{F}_{t_0}$-measurable.
Thus, $\wt{V}|_{\Omega\times[0,t_0]}$ is $\mathcal{F}_{t_0}\otimes\mathcal{B}([0,t_0])$-measurable, i.e., $\wt{V}$ is obviously progressively measurable.

{\em Step 1:} One has
\begin{equation}\label{7.1.2017.1}
\pi_0(V_s) = \sup_{\tau\in\mathcal{T}_s}\, \pi_0(L_\tau)\quad\mbox{for all\ }s\in\Q\cap[0,T] 
\end{equation}
and
\begin{equation}\label{29.12.2016.2}
\pi_t(V_s)\le V_t,\quad P\mbox{-a.s.\ for all\ }t,s\in\Q,\ 0\le t\le s\le T. 
\end{equation}
Indeed, $\pi$ is time consistent, strictly monotone, continuous and by the 
local property of $\pi_s(\cdot)$, the set $\{\pi_s(L_{\tau})\  |\ \tau\in\mathcal{T}_s\}$ is maximum-stable. Consequently, the assertions follow one-to-one from the standard arguments for the linear expectation, see, e.g., Lemma~D.1 and Proposition~D.2 in
\cite{KaratzasShreveMathFi}, where (D.3) is only
evaluated at deterministic and rational-valued stopping times. 

{\em Step 2:} Let us show that there exists a set $\Omega_1\in\mathcal{F}$ with $P[\Omega_1]=1$ such that
\begin{equation}\label{23.12.2016.1}
\wt{V}_t(\omega) = \liminf_{s>t, s\to t}\wt{V}_s(\omega),\ \forall t\in [0,T],\ \omega\in \Omega_1.
\end{equation}
For $t\in \R\setminus \Q$, 
\cref{23.12.2016.1} is satisfied for all $\omega\in\Omega$. Indeed, for all $\omega\in\Omega$ and $q\in\Q$,
one has by definition that $\wt{V}_q(\omega)\le V_q(\omega)$ and thus 
\beao
\liminf_{s>t, s\to t}\wt{V}_s(\omega)
\le \liminf_{s>t, s\in\Q, s\to t}V_s(\omega)=\wt{V}_t(\omega).
\eeao 
On the other hand, for every $\eps>0$ and $m\in\N$, there exists $s_m\in\R$ with $s_m\in(t,t+1/m)$ and $\wt{V}_{s_m}(\omega)\le \liminf_{s>t, s\to t}\wt{V}_s(\omega) + \eps$. To every $s_m$, there belongs a $q_m\in\Q$ with $q_m\in(t,t+2/m)$ and $V_{q_m}(\omega)\le \wt{V}_{s_m}(\omega) + \eps$. This yields $\wt{V}_t(\omega) \le \liminf_{s>t, s\to t}\wt{V}_s(\omega) + 2\eps$. Putting together, we arrive at \cref{23.12.2016.1} for all $\omega\in\Omega$ and $t\in \R\setminus \Q$.
 
Now let $t\in\Q$. One has 
$P$-almost surely,
\begin{equation}\label{22.12.2016.1} 
\sup_{m\in\N}\inf_{s\in(t,t+\frac{1}{m}]\cap\Q}V_s =
\pi_t\Big(\sup_{m\in\N}\inf_{s\in(t,t+\frac{1}{m}]\cap\Q}V_s\Big)\le \sup_{m\in\N}\inf_{s\in(t,t+\frac{1}{m}]\cap\Q}\pi_t(V_s)\le V_t,
\end{equation}
where for the first inequality, monotone convergence is used, and the second inequality holds by \cref{29.12.2016.2}.
On the paths where \cref{22.12.2016.1} holds, it follows that $\wt{V}_t = \liminf_{s>t, s\to t}\wt{V}_s$ by the same reasons as for irrational points. This implies \cref{23.12.2016.1}.

{\em Step 3:}  Let us show that $\wt{V}$ satisfies (i). We start with   deterministic stopping times, i.e., $t,s\in\R$, $0\le t\le s\le T$ and $\tau\equiv s$. Since the case $t=s$ is trivial, assume that $t<s$. 
Using again monotone convergence and \cref{29.12.2016.2}, yields for any $u\in[t,s]\cap\mathbb{Q}$ that 
\[
\pi_u(\wt{V}_s) =  \pi_u\Big(\sup_{m\in\N}\inf_{v\in[s,s+1/m]\cap\mathbb{Q}} V_v\Big)\le \sup_{m\in\N}\inf_{v\in[s,s+1/m]\cap\mathbb{Q}}\pi_u(V_v)\le V_u,\quad P\mbox{-a.s.}.
\]
By the right-continuity of $\pi_\cdot(\wt{V}_t)$, this leads to
\begin{equation}\label{eq:SupMartPropTildeV}
    \pi_t(\wt{V}_s)\le \wt{V}_t,\quad P\mbox{-a.s.}.
\end{equation}
We now show that \cref{eq:SupMartPropTildeV} extends to stopping times, i.e., $\pi_t(\wt{V}_\tau)\le \wt{V}_t$ also holds for any stopping time $\tau\in\mathcal{T}_t$.
We first show that this holds for stopping times with finitely many values in $[t,T]$. Let $\tau$ be such a stopping time, valued in $\{t_0,\ldots,t_N\}$, with 
$t=t_0<t_1<\ldots<t_N.$
Then by translation invariance, that is a special case of Property~3 in 
Proposition~\ref{pro:IndiffValueProperties}, one has
\[
    \pi_{t_{N-1}}(\wt{V}_{\tau}) = \pi_{t_{N-1}}\Big(\sum_{k=1}^N \mathds{1}_{\left\lbrace \tau=t_k\right\rbrace}\wt{V}_{t_k}\Big) =\pi_{t_{N-1}}\big(\mathds{1}_{\left\lbrace \tau=t_N\right\rbrace}\wt{V}_{t_N}\big)+\sum_{k=1}^{N-1}\mathds{1}_{\left\lbrace \tau=t_k\right\rbrace}\wt{V}_{t_k}.
\]
Since $\{\tau=t_N\}=\{\tau>t_{N-1}\}\in \F_{t_{N-1}}$, the local property  and \cref{eq:SupMartPropTildeV} give
\[
    \pi_{t_{N-1}}(\wt{V}_{\tau}) \le \mathds{1}_{\left\lbrace \tau=t_N\right\rbrace} \wt{V}_{t_{N-1}}+\sum_{k=1}^{N-1}\mathds{1}_{\left\lbrace \tau=t_k\right\rbrace}\wt{V}_{t_k} =\mathds{1}_{\left\lbrace \tau\ge t_{N-1}\right\rbrace} \wt{V}_{t_{N-1}}+\sum_{k=1}^{N-2}\mathds{1}_{\left\lbrace \tau=t_k\right\rbrace}\wt{V}_{t_k}.
\]
Using a backward induction in $k=N-1,N-2,\ldots$, we obtain by time-consistency that $\pi_{t_0}(\widetilde{V}_{\tau})\le 
\mathds{1}_{\{\tau\ge t_0\}}\wt{V}_{t_0}$, i.e., $\pi_t(\wt{V}_{\tau})\le\wt{V}_t$. Now for an arbitrary $\tau$ in $\mathcal{T}_t$, there exists a sequence $(\tau_n)_n$ of finitely-valued stopping times $\tau_n$ decreasing to $\tau$.
By \cref{23.12.2016.1}, this implies $\wt{V}_{\tau} \le \liminf_{n\to \infty}\wt{V}_{\tau_n}$.
Then, from monotonicity and continuity of $\pi_t(\cdot)$ (note that this holds under the continuity assumption on the filtration, see, \cref{pro:IndiffValueProperties}), we have
\begin{equation}\label{9.1.2017.1}
    \pi_t(\wt{V}_{\tau}) 
\le \pi_t\Big(\liminf_{n\to \infty}\wt{V}_{\tau_n}\Big)
\le   \liminf_{n\to \infty} \pi_t(\wt{V}_{\tau_n})
    \le \wt{V}_t,\quad P\mbox{-a.s.}.
\end{equation}
This means that $\wt{V}$ satisfies (i). 

{\em Step 4:} Let us show that
\begin{equation}\label{7.1.2017.2}
\pi_0(\wt{V}_t) = \sup_{\tau\in\mathcal{T}_t}\, \pi_0(L_\tau)\quad\mbox{for all\ }t\in [0,T] .
\end{equation}
Indeed, by \cref{9.1.2017.1}, it holds that $\wt{V}_t\ge \pi_t(\wt{V}_\tau)\ge \pi_t(L_\tau)$ $P$-a.s.\ for all 
$\tau\in\mathcal{T}_t$ which implies ``$\ge$'' by time-consistency and monotonicity of $\pi$. On the other hand: from \cref{7.1.2017.1} and $\mathcal{T}_u\subset \mathcal{T}_t$ for all $u\ge t$,
it follows that $\displaystyle\pi_0(\wt{V}_t)\le \liminf_{u\in\Q, u\ge t,\ u\to t}\pi_0(V_u)\le \sup_{\tau\in\mathcal{T}_t}\, \pi_0(L_\tau)$.

{\em Step 5:} Define $\tau^\eps_t:=\inf\{u\in[t,T]\cap\Q\ |\ L_u\ge \wt{V}_u - \eps\}$. Then $\tau^\eps_t\in\mathcal{T}_t$. By \cref{7.1.2017.2} and the right-continuity of $L$, there exist $[t,T]\cap \Q$-valued stopping times~$\sigma_n$ with $\pi_0(L_{\sigma_n})\ge \pi_0(\wt{V}_t) - 1/n$.
We have $L\le \wt{V}$ and $L\le \wt{V} - \eps$ on $[t,\tau^\eps)\cap\Q$. Thus,
\begin{equation*}
\pi_0(L_{\sigma_n}) \le \pi_0\left(\wt{V}_{\sigma_n} - \eps \mathds{1}_{\{\sigma_n<\tau^\eps_t\}}\right).
\end{equation*}
By Step~3, holds $\pi_0(\wt{V}_{\sigma_n})\le \pi_0(\wt{V}_t)$. Putting together: 
$\pi_0(\wt{V}_{\sigma_n})-\pi_0\big(\wt{V}_{\sigma_n} - \eps \mathds{1}_{\{\sigma_n<\tau^\eps_t\}}\big)\rightarrow 0$, as $n\to\infty$. From \cref{pro:ConvOfExpIMpliesConvInProb}, follows 
\begin{equation}\label{23.12.2016.2}
P[\sigma^n<\tau^\eps_t]\longrightarrow 0,\quad n\to\infty.
\end{equation}
In addition one has 
\begin{equation}\label{23.12.2016.3}
\begin{split}\pi_0(\wt{V}_{\sigma_n\wedge\tau^\eps_t}) &\ge \pi_0(L_{\sigma_n}\mathds{1}_{\{\sigma_n\le\tau^\eps_t\}} + \wt{V}_{\tau^\eps_t}\mathds{1}_{\{\tau^\eps_t<\sigma_n\}})\\
&\ge \pi_0(\pi_{\sigma_n\wedge\tau^\eps_t}(L_{\sigma_n}))\\
&=\pi_0(L_{\sigma_n})\ge \pi_0(\wt{V}_t) - \frac1{n},
\end{split}
\end{equation}
where the second inequality holds because $\wt{V}_{\tau^\eps_t}=\sup_{m\in\N}\inf_{s\in[\tau^\eps_t,(\tau^\eps_t+\frac{1}{m})\wedge\sigma_n]\cap\Q}V_s$
on $\{\tau^\eps_t<\sigma_n\}$, $P\left[V_s\ge \pi_s(L_{\sigma_n})\text{ on }\{\sigma_n\ge s\},\ \forall s\in[0,T]\cap\Q\right]=1$ and $\pi_\cdot(L_{\sigma_n})$ is right-continuous.
By \cref{23.12.2016.2}, we have $\wt{V}_{\sigma_n\wedge\tau^\eps_t}\longrightarrow 
\wt{V}_{\tau^\eps_t}$ in probability as $n\to\infty$. Together with 
\cref{23.12.2016.3} and continuity of $\pi_0(\cdot)$, this yields 
$\pi_0(\wt{V}_{\tau^\eps_t})\ge \pi_0(\wt{V}_t)$. On the other hand 
$L_{\tau^\eps_t}\ge \wt{V}_{\tau^\eps_t} -\eps$ and we arrive at
\begin{equation*}
\pi_0(L_{\tau^\eps_t}) \ge \pi_0(\wt{V}_t) - \eps.
\end{equation*}
Let $\tau^\star_t:=\sup_{\eps>0}\tau^\eps_t$.
Since $L$ has no negative jumps, $L_{\tau^\star_t} \ge \lim_{\eps\to 0} L_{\tau^\eps_t}$ and $\tau^\star_t$ is an optimal stopping strategy for $\mathcal{T}_t$, i.e.\ $\pi_0(L_{\tau^\star_t})= \pi_0(\wt{V}_t)$.

One has $\wt{V}_{\tau^\star_t}\ge L_{\tau^\star_t}$ and by Step~3, 
$\pi_0(\wt{V}_{\tau^\star_t})\le \pi_0(\wt{V}_t)$. Thus, $P\big[\wt{V}_{\tau^\star_t}=L_{\tau^\star_t}\big]=1$ by the strict monotonicity of $\pi_0(\cdot).$
Since $L_u=\wt{V}_u$ for some $u\in[t,\tau^\star)$ would lead to a contradiction, we conclude
\begin{equation*}
P\big[\tau^\star_t=\inf\{u\in\R\ |\ u\ge t,\ L_u=\wt{V}_u\}\big]=1
\end{equation*}
and by the usual conditions $\inf\{u\in\R\ |\ u\ge t,\ L_u=\wt{V}_u\}$ is a stopping time. Property (i) and the optimality of $\tau^\star_t$ yield that 
$\wt{V}$ coincides with the RHS of \cref{7.1.2017.3} $P$-a.s..

{\em Step 6:} Putting together, we have shown that $\wt{V}$ is progressively measurable and satisfies both (\ref{7.1.2017.3}) and the properties (i), (ii). Now, we  proceed to right-continuity. First, we show that $\wt{V}$ (itself) is right-continuous $P$-a.s.\ along stopping times.
Let $\tau\in\mathcal{T}_0$ with $\tau<T$ and  $(\tau_n)_{n\in\N}\subset\mathcal{T}_0$ with $\tau_n\downarrow \tau$ for $n\uparrow \infty$. Let us show that  $P[\wt{V}_{\tau_n}\longrightarrow\wt{V}_{\tau}]=1$. By (\ref{23.12.2016.1}), it remains to show that  
\beam\label{26.5.2018.2}
P\left(\limsup_{n\to\infty} \wt{V}_{\tau_n}\le \wt{V}_{\tau}\right)=1.
\eeam
For every $m\in\N$, we consider the debut~$D_m:=\inf\{u>\tau\ |\ |L_u - L_\tau|>1/m\}\wedge T$. By the right-continuity of $L$, one has $D_m>\tau$.  
By the continuity of the filtration, there exists a continuous version of the martingale $t\mapsto E(D_m\ |\ \mathcal{F}_t)$, and thus $D_m$ possesses the announcing sequence~$(T_{m,k})_{k\in\N}\subset \mathcal{T}_0$ given by 
$T_{m,k}:=\inf\{t>\tau\ |\ E(D_m\ |\ \mathcal{F}_t) \le t+1/k\}$. We want to use this to show that
\beam\label{26.5.2018.1}
P(\tau<\sigma_m<D_m)=1\quad\mbox{for some\ }\sigma_m\in\mathcal{T}^{\Q}_0.
\eeam
By a standard exhausting argument, it is sufficient to construct such a $\sigma_m$ on the set~$B:=\{E(D_m\ |\ \mathcal{F}_\tau) > \tau + 1/k_0\}\in\mathcal{F}_\tau$, $k_0\in\N$. 
On $B$, one has $\tau < T_{m,k_0} < T_{m,k_0+1} < T_{m,k_0+2} < \ldots < D_m$. 
We choose $l$ as the minimal integer s.t. $P(B\cap \{\lfloor T_{m,k_0} l+1\rfloor/l> T_{m,k_0+1}\})\le 2^{-k_0}$ and put $\sigma_m=\lfloor T_{m,k_0}l+1\rfloor/l$ on $B\cap\{\lfloor T_{m,k_0} l+1\rfloor/l\le T_{m,k_0+1}\}$. On $B\cap\{\lfloor T_{m,k_0} l+1\rfloor/l > T_{m,k_0+1}\}$, we proceed analogously and determine some rational between $T_{m,k_0+1}$ and $T_{m,k_0+2}$ which is only missed on $B$ with unconditional probability smaller than $2^{-(k_0+1)}$. By the lemma of Borel-Cantelli, this construction leads to a stopping time~$\sigma_m$ satisfying
(\ref{26.5.2018.1}). By (\ref{26.5.2018.1}), there also exists a $\delta_m\in\R_+\setminus\{0\}$ small enough s.t. $P(A_m)\ge 1 -1/m$ where $A_m:=\{\sigma_m\ge \tau + \delta_m\}$.

For $m$ large, by the construction of $\sigma_m$,  the European indifference value of the claim~$V_{\sigma_m}$ shortly after time~$\tau$ is similar to the American one. Thus, we can show (\ref{26.5.2018.2}) by using the right-continuity of the dynamic European indifference price derived in Mania and Schweizer~\cite{ManiaSchweizer05}. In the following, we work out this idea in detail. Let $u\in\Q$. 
On $\{\tau\le u\le\sigma_m\}$, one has 
\beao
V_u \le \pi_u\left(\sup_{\tau\le v\le \sigma_m}L_v\vee V_{\sigma_m}\right)
\le \pi_u(V_{\sigma_m}+2/m) = \pi_u(V_{\sigma_m})+2/m,\quad\mbox{$P$-a.s.,}
\eeao
where the first inequality uses the fact that for all $\sigma\in\mathcal{T}_u$
\beao
\pi_u(L_\sigma) 
= \pi_u(\mathds{1}_{\{\sigma\le\sigma_m\}}L_\sigma + \mathds{1}_{\{\sigma>\sigma_m\}}
\pi_{\sigma_m}(L_\sigma))
& \le & \pi_u\left(\sup_{u\le v\le \sigma_m}L_v\vee V_{\sigma_m}\right)\\
& & \quad\mbox{$P$-a.s.\ on}\ \{\sigma_m\ge u\},
\eeao
and the second inequality holds by (\ref{26.5.2018.1}).
It follows that for all $\delta\in (0,\delta_m]$
\beam\label{25.5.2018.1}
\wt{V}_{\tau_n} 
\le \sup_{u\in\Q, \tau\le u\le \tau + \delta} V_u
& \le & \sup_{u\in\Q, \tau\le u\le \tau + \delta}\pi_u(V_{\sigma_m})+2/m\\
& & \qquad\qquad\quad\mbox{$P$-a.s.\ on}\ A_m\cap \{\tau_n< \tau+\delta\}\nonumber.
\eeam
By the right-continuity of the European indifference price, i.e.,
\beam\label{25.5.2018.3}
\pi_t(V_{\sigma_m})\longrightarrow \pi_{\tau}(V_{\sigma_m})\quad \mbox{$P$-a.s.\ for}\ t\downarrow \tau,
\eeam
the RHS of (\ref{25.5.2018.1}) converges $P$-a.s. to 
$\pi_\tau(V_{\sigma_m})+2/m$ for $\delta\downarrow 0$. This yields
\beam\label{27.5.2018.1}
\limsup_{n\to\infty} \wt{V}_{\tau_n}\le \pi_\tau(V_{\sigma_m}) + 2/m
\quad \mbox{$P$-a.s.\ on\ }A_m.
\eeam

On the other hand, by $\sigma_m>\tau$, one has
\beam\label{25.5.2018.2}
\wt{V}_\tau = \sup_{k\in\N}\inf_{s\in[\tau,\tau+1/k]\cap\Q}V_s
\ge \sup_{k\in\N}\inf_{s\in[\tau,\tau+1/k]\cap\Q}\pi_s(V_{\sigma_m})
= \pi_\tau(V_{\sigma_m}),\ \mbox{$P$-a.s.},
\eeam
where the last equality follows again from (\ref{25.5.2018.3}).
Putting (\ref{27.5.2018.1}) and (\ref{25.5.2018.2}) together, we conclude
\beao
P\left(\limsup_{n\to\infty} \wt{V}_{\tau_n}\le \wt{V}_{\tau} + 2/m\right)\ge P(A_m)\ge 1- 1/m,\quad\forall m\in\N,
\eeao
which implies (\ref{26.5.2018.2}).

{\em Step 7:} Let $\widehat{V}$ be the $P$-optional projection of the bounded process~$\wt{V}$, i.e., $\widehat{V}$ is optional
and $\widehat{V}_\tau = E_\tau[\wt{V}_\tau]$ $P$-a.s.\ for all $\tau\in\mathcal{T}_0$ (see, e.g., Theorem~5.1 of He et al.~\cite{HeWangYan92}). Since $\wt{V}$ is progressively measurable, $\wt{V}_\tau$ is
$\mathcal{F}_\tau$-measurable, and we arrive at  $\widehat{V}_\tau = \wt{V}_\tau$ $P$-a.s.\ for all $\tau\in\mathcal{T}_0$. 
It follows from a section theorem (see, e.g., Theorem~4.7 in \cite{HeWangYan92}) that for every $t\in[0,T]$, the first time $\widehat{V}$ hits $L$ after time~$t$ is a stopping time. 
Then clearly, $\widehat{V}$, $\wt{V}$, and $L$ coincide $P$-a.s.\ at the minimum of this stopping time and $\tau^\star_t$ from Step 5. 
This implies the equality of the two stopping times. 
Since the optional process $\widehat{V}$ is $P$-a.s.\ right-continuous along stopping times, it follows again by a section theorem that it is right-continuous up to evanescence. 
With the usual conditions, one can choose $\widehat{V}$ to be right-continuous everywhere. Uniqueness is obvious.

{\em Step 8:} Let us show that the optional projections of \cref{12.1.2017.1} satisfy (iii).
Let $\sigma\in\mathcal{T}_0$ and $L^1,L^2$ be two payoff processes satisfying $L^1=L^2$ on $[\sigma,T]$, to which we associate $V^1,V^2$ and $\wt{V}^1,\wt{V}^2$ as above. Let $\tau\in\mathcal{T}_0,\ s\in\Q,\ m\in\N$. 
By $\{\tau\ge \sigma,\, \tau\in[s-1/m,s]\}\in\mathcal{F}_s$, the local property of $\pi_s(\cdot)$ implies for $i\in\{1,2\}$,
\begin{equation*}
\mathds{1}_{\left\lbrace\tau\ge \sigma,\, \tau\in\left[s-\frac{1}{m},s\right]\right\rbrace}
\esssup_{\nu\in\mathcal{T}_s}\pi_s( L^i_\nu)
=\esssup_{\nu\in\mathcal{T}_s}\pi_s(\mathds{1}_{\left\{\tau\ge \sigma,\, \tau\in\left[s-\frac{1}{m},s\right]\right\}} L^i_\nu)\quad P\mbox{-a.s},
\end{equation*}
where, by assumption, the RHS does not depend on $i$. Hence we obtain 
\begin{equation}\label{11.1.2017.1}
\mathds{1}_{\left\{\tau\ge \sigma,\, \tau\in\left[s-\frac{1}{m},s\right]\right\}}V^1_s = \mathds{1}_{\left\{\tau\ge \sigma,\, \tau\in\left[s-\frac{1}{m},s\right]\right\}}V^2_s\quad P\mbox{-a.s.}.
\end{equation}
On the other hand, the definition of $\wt{V}$ yields for $i\in\{1,2\}$,
\begin{equation}\label{11.1.2017.2}
\mathds{1}_{\{\tau\ge \sigma\}} \wt{V}^i_\tau = \sup_{m\in\N}\inf_{s\in\Q}\left( \mathds{1}_{\{\tau\ge \sigma\}}\mathds{1}_{\left\{\tau\in\left[s-\frac{1}{m},s\right]\right\}}V^i_s + \infty \mathds{1}_{\{\tau\ge \sigma\}}\mathds{1}_{\left\{\tau\not\in\left[s-\frac{1}{m},s\right]\right\}}\right).
\end{equation}
Putting \cref{11.1.2017.1,11.1.2017.2} together, one obtains  
$\mathds{1}_{\{\tau\ge \sigma\}} \wt{V}^1_\tau = \mathds{1}_{\{\tau\ge \sigma\}} \wt{V}^2_\tau$ $P$-a.s.. By Step~7, one can replace $\wt{V}^1$ and $\wt{V}^2$ by
their optional projections. Then, assertion~(iii) follows again by applying a section theorem.

{\em Step 9:} Assertion~(iv) follows with the same arguments as in Step~8 using the fact that $L=L_\sigma$ on $[\sigma,T]$ implies 
\begin{equation*}
\mathds{1}_{\left\{\tau\ge \sigma,\, \tau\in\left[s-\frac{1}{m},s\right]\right\}} L_\nu
= \mathds{1}_{\left\{\tau\ge \sigma,\, \tau\in\left[s-\frac{1}{m},s\right]\right\}} L_\sigma,\quad \tau\in\mathcal{T}_0,\ s\in\Q,\ m\in\N,\ \nu\in\mathcal{T}_s
\end{equation*}
and, by $\mathcal{F}_s$-measurability, the RHS coincides with its 
$\pi_s$-indifference value. 
\end{proof}
\begin{proof}[Proof of \cref{lem:IneqJitauAndtaun}]
	We show this using the nonlinear Snell envelope of \cref{thm:NonlinSnellNEW}. Consider the 
construction \cref{eq:PayoffProcessesOddEven,eq:OptStoppingTimesOdd,eq:SnellEnvOdd,eq:ApproxNEPOdd}.
By \cref{thm:NonlinSnellNEW}(iv), one has
	$V^{2n+1}_{\tau_{2n}} = L^{2n+1}_{\tau_{2n}}=X_T\mathds{1}_{\{\tau_{2n}=T\}}+Y_{\tau_{2n}}\mathds{1}_{\{\tau_{2n}<T\}}$. Since $X\le Y$ and $V^{2n+1}$ dominates $L^{2n+1}$, we conclude that $X_\tau\mathds{1}_{\{\tau\le\tau_{2n}\}}+Y_{\tau_{2n}}\mathds{1}_{\{\tau>\tau_{2n}\}}\le V^{2n+1}_{\tau\wedge \tau_{2n}}$
for any $\tau\in\mathcal{T}_0$.
	Monotonicity of $\pi^B_0(\cdot)$  and the $\pi^B$-supermartingale
	property of $V^{2n+1}$ (see  \cref{thm:NonlinSnellNEW}(i)) imply
    \begin{equation}\label{eq:FirstJ1Exp}
	\begin{split}
		J_B(\tau,\tau_{2n}) = \pi^B_0\big( X_\tau\mathds{1}_{\{\tau\le\tau_{2n}\}}+Y_{\tau_{2n}}\mathds{1}_{\{\tau>\tau_{2n}\}}\big)
				    \le \pi^B_0\big(V^{2n+1}_{\tau\wedge \tau_{2n}}\big)\le V^{2n+1}_0.
	\end{split}
	\end{equation}
On the other hand, as already observed in \cref{eq:OptStoppingTimesOdd}, 
\begin{equation*}
\pi^B_0(L^{2n+1}_{\tilde{\tau}_{2n+1}}) = \sup_{\tau\in\mathcal{T}_0}\pi^B_0(L^{2n+1}_\tau) = V^{2n+1}_0. 
\end{equation*}
Using the definition of $\tilde{\tau}_{2n+1}$, part 3.\ of \cref{lem:PropSTimes} implies that 
$\{\tilde{\tau}_{2n+1}=\tau_{2n}\} \subset \{\tau_{2n-1} \ge \tau_{2n}\}$. This gives
that $L^{2n+1}_{\tau_{2n+1}}=L^{2n+1}_{\tilde{\tau}_{2n+1}}$, i.e., $\tau_{2n+1}$ is also the maximizer of $\sup_{\tau\in\mathcal{T}_0}\pi^B_0(L^{2n+1}_\tau)$. In addition, by part 2.\ of \cref{lem:PropSTimes},
$\{\tau_{2n+1}=\tau_{2n}\} =  \{\tau_{2n+1}=\tau_{2n}=T\}$, which implies that
\begin{align*}
L^{2n+1}_{\tau_{2n+1}}
& =  X_{\tau_{2n+1}}\mathds{1}_{\{\tau_{2n+1}<\tau_{2n}\}}
 + X_T\mathds{1}_{\{\tau_{2n+1}=\tau_{2n}=T\}}
 + Y_{\tau_{2n}}\mathds{1}_{\{\tau_{2n+1}>\tau_{2n}\}}\\
& =  X_{\tau_{2n+1}}\mathds{1}_{\{\tau_{2n+1}\le\tau_{2n}\}}+Y_{\tau_{2n}}\mathds{1}_{\{\tau_{2n+1}>\tau_{2n}\}}.
\end{align*} 
Putting together, one obtains
\begin{equation}\label{eq:SecondJ1Exp}
\begin{split}
	J_B(\tau_{2n+1},\tau_{2n}) &= \pi^B_0\big( X_{\tau_{2n+1}}\mathds{1}_{\{\tau_{2n+1}\le\tau_{2n}\}}+Y_{\tau_{2n}}\mathds{1}_{\{\tau_{2n+1}>\tau_{2n}\}}\big)\\
				   &=\pi^B_0\big( L^{2n+1}_{\tau_{2n+1}}\big)=V^{2n+1}_0. 
\end{split}
\end{equation} 
Combining \cref{eq:FirstJ1Exp,eq:SecondJ1Exp} yields the first inequality in \cref{eq:ApproximateNEP}. The second inequality of \cref{eq:ApproximateNEP} is obtained similarly but simpler, since $t\mapsto R(\tau_{2n+1},t)$ is already 
c\`adl\`ag and $L^{2n+2}_t = -R(\tau_{2n+1},t)$ for all $t\in[0,T]$.  
\end{proof}
\begin{proof}[Proof of \cref{lem:1stEqTowardsNEP}]
Let $\tau\in\mathcal{T}_0$.\\
	Part 1: By $\tau_{2n}\ge \tau^*_2$ for all $n\in\N$, the right-continuity
	of $t\mapsto R(\tau,t)$, and the continuity of $\pi^B_0(\cdot)$, it follows for $J_B(\tau,\tau_{2n}):= \pi^B_0\big(R(\tau,\tau_{2n})\big)$ that
	\[
	\lim_{n\to \infty}J_B(\tau,\tau_{2n}) = J_B(\tau,\tau^*_2). 
	\]
	Part 2:  Since  $t\mapsto R(t,\tau)$ may not be right-continuous, we cannot argue as in Part~1. Instead, we apply the arguments of Part~1 to the right limit $t\mapsto -R(t+,\tau)$ and obtain
\[
	\lim_{n\to \infty}J_A(\tau_{2n+1},\tau) 
	= \pi^A_0\big(-X_{\tau^*_1}\mathds{1}_{\{\tau>\tau^*_1\}}-Y_{\tau}\mathds{1}_{\{\tau\le\tau^*_1\}}\big). 
	\]	
But, under the assumption that $P[\tau=\tau^*_1<T]=0$, the RHS coincides with 	$J_A(\tau^*_1,\tau)$, and we are done.
\end{proof}

\begin{proof}[Proof of \cref{lem:2ndEqTowardsNEP}]
	We first show part 2. By definition, one has 
	$$
	J_A(\tau_{2n-1},\tau_{2n})=\pi^A_0(G^n),\quad\mbox{where\ }G^n:= -X_{\tau_{2n-1}}\mathds{1}_{\{\tau_{2n}\ge\tau_{2n-1}\}}-Y_{\tau_{2n}}\mathds{1}_{\{\tau_{2n}<\tau_{2n-1}\}}.
	$$
One can write $G^n=H^n+\eta^n$, where 
\begin{equation*}
\eta^n:= (X_{\tau^*_1}-Y_{\tau^*_1})\mathds{1}_{\{\tau_{2n}<\tau_{2n-1},\, \tau^*_1=\tau^*_2\}}
\end{equation*}
and	
\begin{equation}\label{11.4.2017.1}
	\begin{split}
	H^n &= -X_{\tau_{2n-1}}\mathds{1}_{\{\tau_{2n}\ge\tau_{2n-1},\, \tau^*_1<\tau^*_2\}}-Y_{\tau_{2n}}\mathds{1}_{\{\tau_{2n}<\tau_{2n-1},\, \tau^*_1<\tau^*_2\}}\\
	 &-X_{\tau_{2n-1}}\mathds{1}_{\{\tau_{2n}\ge\tau_{2n-1},\, \tau^*_1>\tau^*_2\}}-Y_{\tau_{2n}}\mathds{1}_{\{\tau_{2n}<\tau_{2n-1},\, \tau^*_1>\tau^*_2\}}\\
	 & -X_{\tau_{2n-1}}\mathds{1}_{\{\tau^*_1=\tau^*_2\}}
	 + (X_{\tau_{2n-1}}-X_{\tau^*_1}+Y_{\tau^*_2} -Y_{\tau_{2n}} )\mathds{1}_{\{\tau_{2n}<\tau_{2n-1},\, \tau^*_1=\tau^*_2\}}.
\end{split}
\end{equation}
By the condition \cref{eq:XleY},
$\eta^n$ is nonpositive. It is the negative of the hypothetical limiting penalty on the event that for the approximating stopping times, the seller 
stops before the buyer (hence, has to pay the  
penalty), whereas the limiting stopping times coincide (i.e., actually no penalty has to be paid in the limit).
After correcting for this term, which potentially produces a discontinuity, 
it is easy to see from \cref{11.4.2017.1} that by the right-continuity of $X$ and 
$Y$, $H^n = G^n - \eta^n$ converges pointwise to   
\begin{equation*}
H:= -X_{\tau^*_1}\mathds{1}_{\{\tau^*_1\le \tau^*_2\}}-Y_{\tau^*_{2}}\mathds{1}_{\{\tau^*_1> \tau^*_{2}\}}.	
\end{equation*}
For this, note in addition that $\mathds{1}_{\{\tau_{2n}<\tau_{2n-1},\, \tau^*_1>\tau^*_2\}}\rightarrow \mathds{1}_{\{\tau^*_1>\tau^*_2\}}$,
 $\mathds{1}_{\{\tau_{2n}\ge \tau_{2n-1},\, \tau^*_1<\tau^*_2\}}\rightarrow \mathds{1}_{\{\tau^*_1<\tau^*_2\}}$, $\mathds{1}_{\{\tau_{2n}<\tau_{2n-1},\, \tau^*_1<\tau^*_2\}}\rightarrow 0$
 and $\mathds{1}_{\{\tau_{2n}\ge\tau_{2n-1},\, \tau^*_1>\tau^*_2\}}\rightarrow 0$, as $n\rightarrow\infty$.
Then, the continuity of $\pi^A_0(\cdot)$ and the uniform boundedness of $X,Y$ give 
$$
\pi^A_0\big(H^n\big)\longrightarrow 
\pi^A_0\big(H\big) = J_A(\tau^*_1,\tau^*_2)\quad \text{as }n\to \infty. 
$$
This already implies that 
\begin{equation}\label{11.4.2017.2}
\limsup_{n\to\infty}J_A(\tau_{2n-1},\tau_{2n})\le J_A(\tau^*_1,\tau^*_2).
\end{equation}
On the other hand,
\cref{lem:IneqJitauAndtaun} gives 
\begin{equation}\label{eq:ToObtainLimitPricesiszero}
J_A(\tau_{2n-1},\tau)\le J_A(\tau_{2n-1},\tau_{2n})\quad \text{for any } \tau\in\mathcal{T}_0.
\end{equation}
Let $\hat{\tau}$ be defined by  $\hat{\tau}:= \tau^*_2\mathds{1}_{\{\tau^*_2<\tau^*_1\}}+T\mathds{1}_{\{\tau^*_2\ge \tau^*_1\}}$. Then, $\hat{\tau}$ is a stopping time satisfying 
$P[\hat{\tau}=\tau^*_1<T]=0$, and part 2.\ of \cref{lem:1stEqTowardsNEP} implies 
\begin{equation}\label{11.4.2017.3}
\lim_{n\to \infty}J_A(\tau_{2n-1},\hat{\tau})= J_A(\tau^*_1,\hat{\tau}) = J_A(\tau^*_1,\tau^*_2).
\end{equation}
Putting \cref{11.4.2017.2}, \cref{eq:ToObtainLimitPricesiszero}
for $\tau=\hat{\tau}$, and \cref{11.4.2017.3} together yields 
\begin{equation*}
\lim_{n\to\infty} J_A(\tau_{2n-1},\tau_{2n}) = J_A(\tau^*_1,\tau^*_2).	
\end{equation*}
We now turn to the proof of part.\ 1 of the lemma. First note that it is already 
shown that
$$
\lim_{n\to \infty}\pi^A_0(H^n+\eta^n)= \pi^A_0(H) = \lim_{n\to \infty}\pi^A_0(H^n).
$$ 
Hence, \cref{pro:ConvOfExpIMpliesConvInProb} implies that
\begin{equation}\label{eq:H4nto0}
\eta^n \longrightarrow 0\quad \text{ in probability as } n\to \infty.	
\end{equation}
This means that if the hypothetical limiting penalty does not vanish, then 
the probability that the seller stops shortly before the buyer tends to zero.  
The intuition behind this is that on the event that the limiting stopping times coincide,
such an action by the seller cannot be a best response since the buyer stops shortly afterwards without receiving the penalty. 
Part 1.\  of \cref{lem:PropSTimes} yields 
$\{\tau_{2n}<\tau_{2n+1}\}\subset \{\tau_{2n}<\tau_{2n-1}\}$, hence by \cref{eq:H4nto0} this leads to 
\[
\tilde{\eta}^n:= (Y_{\tau^*_1}-X_{\tau^*_1})\mathds{1}_{\{\tau_{2n}<\tau_{2n+1},\, \tau^*_1=\tau^*_2\}}\longrightarrow 0\quad \text{ in probability as } n\to \infty.	
\]
Now, the proof follows as in part 2, with an analogue decomposition of the payoff
$X_{\tau_{2n+1}}\mathds{1}_{\{\tau_{2n}\ge\tau_{2n+1}\}}+Y_{\tau_{2n}}\mathds{1}_{\{\tau_{2n}<\tau_{2n+1}\}}$ and indifference valuation~$\pi^B_0$ instead of $\pi^A_0$.
\end{proof}

\section*{Acknowledgments}

We would like to thank two anonymous referees
for their valuable comments and suggestions from which the manuscript greatly benefited.

\bibliographystyle{plain}
\bibliography{kentia.kuehn.references}
\end{document}